\documentclass[11pt,a4paper]{article}  
\usepackage{amsmath}
\usepackage[english]{babel}
\usepackage[utf8]{inputenc}
\usepackage{amsfonts}
\usepackage{amsthm}
\usepackage{esint}
\usepackage{authblk}
\usepackage{color}
\numberwithin{equation}{section}
\usepackage{geometry}
\newtheorem{thm}{Theorem}[section]

\newtheorem{prop}[thm]{Proposition}
\newtheorem{cor}[thm]{Corollary}
\newtheorem{lem}[thm]{Lemma}

\theoremstyle{remark}
\newtheorem{rmk}[thm]{Remark}
\theoremstyle{definition}
\newtheorem{defn}{Definition}[section]

\DeclareMathOperator{\N}{\mathbb{N}}
\DeclareMathOperator{\R}{\mathbb{R}}
\DeclareMathOperator{\cF}{\mathcal{F}}

\DeclareMathOperator{\cM}{\mathcal{M}}

\DeclareMathOperator{\cP}{\mathcal{P}}
\DeclareMathOperator{\cR}{\mathcal{R}}
\DeclareMathOperator{\cU}{\mathcal{U}}
\DeclareMathOperator{\cA}{\mathcal{A}}

\DeclareMathOperator{\cI}{\mathcal{I}}
\DeclareMathOperator{\cH}{\mathcal{H}}

\DeclareMathOperator{\W}{\mathcal{W}}

\DeclareMathOperator{\divg}{div}
\DeclareMathOperator{\diam}{diam}
\DeclareMathOperator{\pL}{\widetilde{\Delta}_p}
\DeclareMathOperator{\iL}{\widetilde{\Delta}_\infty}
\DeclareMathOperator{\cpL}{\Delta_p}
\DeclareMathOperator{\ciL}{\Delta_\infty}

\providecommand{\keywords}[1]{\textbf{\textit{Keywords---}} #1\\}
\providecommand{\subclass}[1]{\textbf{\textit{MSC 2010 Classification---}} #1\\}

\newcommand{\Norm}[2]{\left\Vert #1 \right\Vert_{#2}}

\newcounter{daggerfootnote}

\title{The orthotropic $p$-Laplace eigenvalue problem of Steklov type  as $p\to+\infty$ }
\author{Giacomo Ascione}
\author{Gloria Paoli}
\date{}
\affil{\footnotesize Dipartimento di Matematica e Applicazioni ``Renato Caccioppoli''\\
		Universit\`a degli Studi di Napoli Federico II\\
		Via Cintia, Complesso Universitario di Monte S. Angelo\\
		80126 Napoli, Italy\\
	giacomo.ascione@unina.it, \ gloria.paoli@unina.it} 
	
\begin{document}
	\maketitle

	\begin{abstract}
		We study the Steklov eigenvalue problem for the $\infty-$orthotropic Laplace operator defined on convex sets of $\mathbb{R}^N$, with $N\geq2$,  considering the limit for $p\to+\infty$ of the Steklov problem for the $p-$orthotropic Laplacian. We find a limit problem that is satisfied in the  viscosity sense  and a geometric characterization of the first non trivial eigenvalue. Moreover, we prove Brock-Weinstock and Weinstock type inequalities among convex sets, stating that the ball in a suitable norm maximizes the first non trivial eigenvalue for the Steklov $\infty-$orthotropic Laplacian, once we fix the volume or the anisotropic perimeter.
	\end{abstract}
\keywords{Viscosity solution, Weinstock inequality, isodiametric inequality, non-linear operator}
\subclass{Primary: 35D40, 35P30; Secondary: 35J66}

\section{Introduction}

Let $\Omega$ be an open, bounded and convex set in $\mathbb{R}^N$, with $N\geq2$, and let $p>1$. We consider  the following operator, called the orthotropic $p-$Laplace operator, sometimes also called pseudo $p-$Laplacian,
\begin{equation*}
	\widetilde{\Delta}_p u=\sum_{i=1}^{N}\left( |u_{x_i}|^{p-2} u_{x_i}\right)_{x_i},
\end{equation*}
where  $u_{x_j}$ is the partial derivative of $u$ with respect to $x_j$, and we study the limit problem, as $p\to+\infty$, of the Steklov problem 
\begin{equation}\label{Stkprob_intro}
	\begin{cases}
		-\widetilde{\Delta}_p u=0 & \mbox{on } \Omega\\
		\sum_{j=i}^{N}|u_{x_j}|^{p-2}u_{x_j}\nu_{\partial \Omega}^j=\sigma|u|^{p-2}u\rho_p & \mbox{on }\partial \Omega,
	\end{cases}
\end{equation}
where $\nu_{\partial \Omega}=(\nu^1_{\partial \Omega}, \dots, \nu^N_{\partial \Omega})$ is the outer normal of $\partial \Omega$,  $\rho_p(x)=\Norm{\nu_{\partial \Omega}(x)}{\ell^{p'}}$, $p'$ is the coniugate exponent of $p$ and 
\begin{equation}\label{norm}
	\Norm{x}{\ell^p}^p=\sum_{j=1}^{N}|x^j|^p.
\end{equation}
The real number $\sigma$ is called Steklov eigenvalue whenever problem \eqref{Stkprob_intro} admits a non-null solution. In particular, problem \eqref{Stkprob_intro} has been investigated  in  \cite{brasco2013anisotropic}. Here, it is proven that these eigenvalues form at least a countably infinite sequence of positive numbers diverging at infinity, where the first eigenvalue is $0$ and corresponds to constant eigenfunctions.
Denoting by $\Sigma_p^p(\Omega)$  the first non-trivial eigenvalue of \eqref{Stkprob_intro}, the following variational characterization is shown (see \cite{brasco2013anisotropic}):
\begin{equation}\label{varcar_i}
	\Sigma_p^p(\Omega)=\min \left\{\frac{\int_{\Omega} \Norm{\nabla u}{\ell^p}^pdx}{\int_{\partial \Omega}|u|^p\rho_pd\cH^{N-1}}, \ u \in W^{1,p}(\Omega), \ \int_{\partial \Omega}|u|^{p-2}u\rho_pd\cH^{N-1}=0\right\}.
\end{equation}
Let us observe that the value $\Sigma_p^p(\Omega)$ represents the optimal constant in the weighted trace-type inequality
\begin{equation*}
	\int_{\Omega} \Norm{\nabla u}{\ell^p}^pdx\ge \Sigma_p^p(\Omega)\int_{\partial \Omega}|u|^p\rho_pd\cH^{N-1}
\end{equation*}
in the class of Sobolev functions $u \in W^{1,p}(\Omega)$, such that \begin{equation*}\int_{\partial \Omega}|u|^{p-2}u\rho_pd\cH^{N-1}=0.\end{equation*}
By the way, we recall that the orthotropic $p-$Laplacian was considered in \cite{Lions, Visik,Visik2}; for $p=2$ it coincides with the  Laplacian, but for $p\neq 2$ it differs from the usual $p-$Laplacian, that is defined as $\Delta_pu:={\rm div}\left(  |\nabla u|^{p-2} \nabla u \right)$.  The orthotropic $p-$ Laplacian can be considered indeed as an anisotropic operator, associated to the Finsler norm  \eqref{norm}. Let us recall that for this operator an isoperimetric inequality concerning the first Dirichlet eigenvalue has been discussed in the planar case in \cite{bognar1993lower,bognar2004isoperimetric}.
In this work we  focus our attention on the limit operator $\widetilde{\Delta}_{\infty} u$,  the so-called orthotropic $\infty$-Laplace operator, that can also be defined, see for example \cite{belloni2004pseudo,rossi2007optimal},  as 
\begin{equation*}\label{infty}
	\iL u(x)= \sum_{j \in I(\nabla u(x))}u_{x_j}^2(x)u_{x_j,x_j}(x),
\end{equation*}
where 
\begin{equation*}
	I(x):=\{j \le N: \ |x_j|=\Norm{x}{\ell^\infty}\}
\end{equation*}
and
\begin{equation*}
	\Norm{x}{\ell^\infty}=\max_{j=1,\dots, N}|x^j|.
\end{equation*}

We are inspired by the results given in \cite{garcia2006steklov}, where the authors study the Steklov eigenvalue problem for the $\infty-$Laplacian $\Delta_{\infty}$,  given by 
\begin{equation*}
	\Delta_{\infty}u=\sum_{i,j=1}^{N}u_{x_j}u_{x_i}u_{x_jx_i}.
\end{equation*}
This operator was also studied for example in \cite{espositoNeumann}, with Neumann boundary conditions, and in \cite{rossi_saintier} for mixed Dirichlet and Robin boundary conditions.

In particular, we find a limit eigenvalue problem of \eqref{Stkprob_intro} that is satisfied in a viscosity sense and we show that we can pass to the limit in the variational caracterization \eqref{varcar_i}. More precisely, we prove the following result.
\begin{thm}
	Let $\Omega$ be a bounded open convex set. It holds
	\begin{equation*}
		\lim_{p \to +\infty}\Sigma_p(\Omega)=\Sigma_\infty(\Omega)=\frac{2}{\diam_1(\Omega)},
	\end{equation*} 
	where $\diam_1(E):=\sup_{x,y\in E}||x-y||_{\ell^1}$.\\
	Moreover, if $\partial \Omega$ is $C^1$, we denote by $u_{2,p}$ an eigenfunction of \eqref{Stkprob_intro} of eigenvalue $\Sigma^p_p(\Omega)$ satisfying the normalization condition
	$$\frac{1}{V(\Omega)}\int_{\partial \Omega}|u_{2,p}|^p\rho_p d\cH^{N-1}=1,$$
	where $V(\cdot)$ is the volume. Then there exists a sequence $p_i\to+\infty$ such that $u_{2,p_i}$ converges uniformly in $\overline{\Omega}$ to $u_{2,\infty}$, that is solution of
	\begin{equation}\label{viscprob}
		\begin{cases}
			-\widetilde{\Delta}_{\infty} u=0 & \mbox{on } \Omega\\
			\Lambda(x,u,\nabla u)=0 & \mbox{on }\partial \Omega
		\end{cases}
	\end{equation}
	in the viscosity sense, where, for $(x,u,\eta) \in \partial \Omega \times \R \times \R^N$, we set
	\begin{equation*}
		\Lambda(x,u,\eta)=\begin{cases}
			\min\big\{  \Norm{\eta}{\ell^\infty}  -\Sigma_{\infty}(\Omega)|u|\;, \;\sum_{j \in I(\eta)}\eta_{j}\nu^j_{\partial\Omega}(x)  \big \} & \mbox{if } u>0\\
			\max\big\{ \Sigma_{\infty}(\Omega)|u|- \Norm{\eta}{\ell^\infty}, \sum_{j \in I(\eta)}\eta_{j}\nu^j_{\partial\Omega}(x)  \big \} & \mbox{if } u<0\\
			\sum_{j \in I(\eta)}\eta_{j}\nu^j_{\partial\Omega}(x) & \mbox{if } u=0.
		\end{cases}
	\end{equation*}
\end{thm}

We  observe that, since the first eigenvalue of \eqref{Stkprob_intro}
is $0$ with constant eigenfunction, we can trivially pass to the limit and obtain that the first eigenvalue of \eqref{viscprob} is also $0$ with constant associated eigenfunction.

The last part of this work is dedicated to the proof of Brock-Weinstock and Weinstock type inequalities for the orthotropic $p-$Laplacian, possibly with $p=\infty$. We will use the following notation to denote respectively  the unit ball and  the anisotropic perimeter  with respect to the $\ell^p$ norm, for $p\in(1,\infty]$, 
\begin{equation*}
	\mathcal{W}_{p}=\{x\in\mathbb{R}^N\;|\Norm{x}{\ell^p}\leq 1\};
\end{equation*}
\begin{equation*}
	\cP_p(\Omega):=\int_{\partial \Omega}\rho_p(x)d\cH^{N-1}(x). \quad 
\end{equation*}
In \cite{brasco2013anisotropic}, a Brock-Weinstock type inequality of the form
\begin{equation}\label{brock}
	\Sigma^p_p(\Omega)\leq \left( \dfrac{V(\mathcal{W}_p)}{V(\Omega)}  \right)^{\frac{p-1}{N}}
\end{equation}
is proven. We recall that the Euclidean version of the Brock-Weinstock inequality was proven in \cite{brock2001isoperimetric} and its quantitative version in \cite{brasco2012spectral}. Let us also recall that (up to our knowledge) we cannot write inequality \eqref{brock} in  a fully scaling invariant form, except for $p=2$, since it is still an open problem to determine whether $\Sigma_p^p(\mathcal{W}_p)=1$ or not for $p \not = 2$, as conjectured in \cite{brasco2013anisotropic}.\\
We improve inequality \eqref{brock}, including also the anisotropic perimeter. More precisely, we obtain the following result.
\begin{thm}
	Let $\Omega \subset \R^N$ be an open bounded convex set and $p>1$. Consider $q\ge 0$ and $r \in [0,N]$ such that $\frac{p}{N}=q+\frac{r}{N}$.
	Then, we have
	\begin{equation}\label{Weins_p_intro0}
		\Sigma_{p}^p(\Omega)\cP_p(\Omega)^{\frac{r-1}{N-1}}V(\Omega)^{q}\le\cP_p(\W_p)^{\frac{r-1}{N-1}}V(\W_p)^{q}.
	\end{equation} 
\end{thm}
This leads to the following Weinstock-type inequality for $p \le N$:
\begin{equation}\label{Weins_p_intro}
	\Sigma_{p}^p(\Omega)\cP_p(\Omega)^{\frac{p-1}{N-1}}\le\cP_p(\W_p)^{\frac{p-1}{N-1}}.
\end{equation}
We observe that, in the case $p=2$, inequality \eqref{Weins_p_intro} is given by
\begin{equation*}
	\Sigma_{2}^2(\Omega)\cP_2(\Omega)^{\frac{1}{N-1}}\le 	 \Sigma_{2}^2(\W_2)\cP_2(\W_2)^{\frac{1}{N-1}}
\end{equation*}
(since $\Sigma_2^2(\W_2)=1$) and it has been proven in \cite{weinstock1954inequalities} in the case $N=2$ for general simply connected sets and generalized for $N>2$ in \cite{bucur2017weinstock}, under the convexity constraint. A quantitative version of such inequality has been achieved in \cite{gavitone2019quantitative}. Both the papers \cite{bucur2017weinstock,gavitone2019quantitative} rely on a particular isoperimetric inequality that has been generalized in \cite{paoli2019anisotropic}. This particular isoperimetric inequality will be used in this work to prove inequality \eqref{Weins_p_intro0}.\\
Concerning the orthotropic $\infty$-Laplacian, the characterization of $\Sigma_\infty(\Omega)$ leads to the following results.
\begin{thm}
	For any bounded open convex set $\Omega\subseteq \R^N$, it holds
	\begin{align*}\label{Wein_infty_intro}
		\Sigma_{\infty}(\Omega)V(\Omega)^{\frac{1}{N}}\le \Sigma_{\infty}(\W_1)V(\W_1)^{\frac{1}{N}}.
	\end{align*}
	Moreover, if $N=2$, it also holds
	\begin{equation*}\label{Wein_infty_intro}
		\Sigma_{\infty}(\Omega)P_\infty(\Omega)\le \Sigma_\infty(\W_1)P(\W_1).
	\end{equation*}
\end{thm}
In particular, this means that the Brock-Weinstock inequality still holds for the orthotropic $\infty$-Laplacian with maximizer of the first non trivial Steklov eigenvalue under volume constraint given by $\W_1$ and, in case $N=2$, we also recover Weinstock inequality with the same maximizer $\W_1$ under anisotropic perimeter constraint.


The  paper is organized as follows. In Section $2$ we have collected some useful notations and some known results about the orthotropic $p-$Laplacian.  In Section $3$ we define viscosity solutions of problem \eqref{Stkprob_intro} and we prove that every continuous weak solution is a viscosity solution. In Section $4$ we give the definition of orthotropic $\infty-$Laplacian and in Section  $5$  we derive the limit eigenvalue and the limit equation as $p\to+\infty$.
Finally, in Section $6$,  we discuss Brock-Weinstock and Weinstock type inequalities for the orthotropic $p-$Laplacian, possibly $p=\infty$.

\section{The $p-$orthotropic Laplace eigenvalue with Steklov boundary condition: definitions and notations. }\label{np}
Fix $p>1$ and  an open bounded convex set $\Omega \subseteq \R^N$ and  consider the Steklov problem for the orthotropic  $p$-Laplacian operator on $\Omega$, sometimes called pesudo $p$-Laplacian, as studied in \cite{brasco2013anisotropic}, that is 
\begin{equation}\label{Stkprob}
	\begin{cases}
		-\widetilde{\Delta}_p u=0 & \mbox{on } \Omega\\
		\sum_{j=i}^{N}|u_{x_j}|^{p-2}u_{x_j}\nu_{\partial \Omega}^j=\sigma|u|^{p-2}u\rho_p & \mbox{on }\partial \Omega,
	\end{cases}
\end{equation}
where $u_{x_j}$ is the partial derivative of $u$ with respect to $x_j$, $\nu_{\partial \Omega}=(\nu^1_{\partial \Omega}, \dots, \nu^N_{\partial \Omega})$ is the outer normal of $\partial \Omega$, $\rho_p(x)=\Norm{\nu_{\partial \Omega}(x)}{\ell^{p'}}$, $p'$ is the coniugate exponent of $p$, and
\begin{equation*}
	\pL u=\divg\left(\cA_p(\nabla u)\right), \qquad \cA_p(\nabla u)=\left(|u_{x_1}|^{p-2}u_{x_1},\dots,|u_{x_N}|^{p-2}u_{x_N}\right).
\end{equation*}
We use the following notation:  for any $x \in \R^N$ and $p \ge 1$ we set
\begin{equation*}
	\Norm{x}{\ell^p}^p=\sum_{j=1}^{N}|x^j|^p,
\end{equation*}
while for $p=\infty$ we have
\begin{equation*}
	\Norm{x}{\ell^\infty}=\max_{j=1,\dots, N}|x^j|.
\end{equation*}
Solutions of \eqref{Stkprob} are to be interpreted in the weak sense; we recall here the definition of weak solution. 
\begin{defn}\label{weak_solution}
	Let $u\in W^{1,p}(\Omega)$. We say that $u$ is a weak solution of \eqref{Stkprob} if 
	\begin{equation*}
		\int_{\Omega} \langle \cA_p(\nabla u), \nabla \varphi\rangle dx=\sigma\int_{\partial \Omega} |u|^{p-2}u\varphi \rho_p d\cH^{N-1} \qquad \forall \varphi \in W^{1,p}(\Omega).
	\end{equation*}
\end{defn}
It has been shown in \cite[Section $4$]{brasco2013anisotropic} that the Steklov problem \eqref{Stkprob} admits a non-decreasing sequence of eigenvalues 
\begin{equation*}
	0=\sigma_{1,p}(\Omega)<\sigma_{2,p}(\Omega)\le \cdots,
\end{equation*}
where the first eigenvalue is trivial for any $p >1$ and corresponds to constant eigenfunctions. We denote the first non-trivial eigenvalue $\sigma_{2,p}(\Omega)=:\Sigma_p^p(\Omega)$. In \cite{brasco2013anisotropic} the following variational characterization of $\Sigma_p^p(\Omega)$ is shown:
\begin{equation}\label{varcar}
	\Sigma_p^p(\Omega)=\min \left\{\frac{\int_{\Omega} \Norm{\nabla u}{\ell^p}^pdx}{\int_{\partial \Omega}|u|^p\rho_p(x)d\cH^{N-1}}, \ u \in W^{1,p}(\Omega), \ \int_{\partial \Omega}|u|^{p-2}u\rho_p(x)d\cH^{N-1}=0\right\}.
\end{equation}
Finally, we observe that, for $C^2$ functions, we can rewrite the orthotropic $p$-Laplace operator in such a way to explicitly see where the second derivatives come into play:
\begin{equation*}\label{pLC2}
	\pL u=\sum_{j=1}^{N}(p-1)|u_{x_j}|^{p-2}u_{x_j,x_j}.
\end{equation*}
\section{Viscosity solutions of the $p$-orthotropic Steklov problem}
In the following we need to work with viscosity solutions to the Steklov problem \eqref{Stkprob}. Let us consider in this section $\Omega$ a $C^1$ open bounded convex subset of $\R^N$. Thus, we denote
\begin{equation*}
	F_p:(\xi,X) \in \R^N \times \R^{N \times N} \mapsto -\sum_{j=1}^{N}(p-1)|\xi_j|^{p-2}X_{j,j}
\end{equation*} 
and
\begin{equation*}
	B_p:(\sigma,x,u,\xi) \in \R \times \partial\Omega \times \R \times \R^N \mapsto \sum_{j=1}^{N}|\xi_j|^{p-2}\xi_j\nu_{\partial \Omega}^j(x)-\sigma |u|^{p-2}u \rho_p(x).
\end{equation*}
Following \cite{garcia2006steklov}, the Steklov problem \eqref{Stkprob} can be \textit{formally} rewritten as
\begin{equation}\label{Stekrev}
	\begin{cases}
		F_p(\nabla u,\nabla^2u)=0, & \mbox{on }\Omega \\
		B_p(\sigma,x,u,\nabla u)=0, & \mbox{on }\partial \Omega.
	\end{cases}
\end{equation}
As a consequence,  the functions $F_p$ and $B_p$ can be used to define viscosity solutions for the Steklov problem \eqref{Stkprob} (see, for instance, \cite{katzourakis2014introduction}).
\begin{defn} 
	Let $u$ be a lower (upper) semi-continuous function on the closure  $\overline{\Omega}$ of $\Omega$ and $\Phi \in C^2(\overline{\Omega})$. We say that $\Phi$ is \textbf{touching from below} (\textbf{above}) $u$ in $x_0 \in \overline{\Omega}$ if and only if $u(x_0)-\Phi(x_0)=0$ and $u(x)>\Phi(x)$ ($u(x)<\Phi(x)$) for any $x \not = x_0$ in $\overline{\Omega}$.\\
	A lower (upper) semi-continuous function $u$ on $\overline{\Omega}$ is said to be a \textbf{viscosity supersolution} (\textbf{subsolution}) of \eqref{Stekrev} if for any function $\Phi \in C^2(\overline{\Omega})$ touching from below (above) $u$ in $x_0 \in \overline{\Omega}$ one has
	\begin{align*}
		F_p(\nabla \Phi(x_0),\nabla^2 \Phi(x_0)) \ge (\le) 0 && x_0 \in \Omega;\\
		\max\{F_p(\nabla \Phi(x_0), \nabla^2 \Phi(x_0)), B_p(\sigma,x_0,\Phi(x_0),\nabla \Phi(x_0))\} \ge (\le) 0 && x_0 \in \partial \Omega.
	\end{align*}
	Finally, we say that a continuous function $u$ on $\overline{\Omega}$ is a \textbf{viscosity solution} if it is both viscosity subsolution and supersolution.
\end{defn}
We need the following technical Lemma, whose proof is given in \cite[Section $10$]{lindqvist2017notes}.
\begin{lem}\label{lemp}
	Let $n \in \N$ and $x,y \in \R^n$. For $p \ge 2$ we have
	\begin{equation*}
		\langle |x|^{p-2}x-|y|^{p-2}y,x-y\rangle \ge 2^{2-p}|x-y|^p.
	\end{equation*}
\end{lem}

Now we are ready to show the following result, which is the $p$-orthotropic version of \cite[Lemma $2.1$]{garcia2006steklov}.
\begin{prop}\label{weak_is_viscosity}
	Fix $p \ge 2$. Let $u$ be a weak solution of the Steklov problem \eqref{Stkprob} that is continuous in $\overline{\Omega}$. Then it is a viscosity solution of \eqref{Stekrev}.
\end{prop}
\begin{proof}
	Let us show that $u$ is a viscosity supersolution of \eqref{Stekrev}, since for the subsolution the proof is analogous. Consider $\Phi \in C^2(\overline{\Omega})$ touching from below $u$ in $x_0 \in \overline{\Omega}$. Let us first consider $x_0 \in \Omega$. We want to show that
	\begin{equation*}
		F_p(\nabla \Phi(x_0), \nabla^2 \Phi(x_0))\ge 0,
	\end{equation*}
	thus we suppose by contradiction that
	\begin{equation*}
		F_p(\nabla \Phi(x_0), \nabla^2 \Phi(x_0))< 0.
	\end{equation*}
	Since $\Phi \in C^2$, there exists a radius $r>0$ such that for any $x \in B_r(x_0)$ it holds
	\begin{equation*}
		F_p(\nabla \Phi(x), \nabla^2 \Phi(x))< 0.
	\end{equation*}
	Consider then 
	\begin{equation*}
		m=\inf_{x \in \partial B_r(x_0)}|u(x)-\Phi(x)|=\inf_{x \in \partial B_r(x_0)}(u(x)-\Phi(x))
	\end{equation*} 
	and define $\Psi(x)=\Phi(x)+\frac{m}{2}$. Since $\Psi$ and $\Phi$ differ only by a constant, $\nabla \Psi=\nabla \Phi$ and $\nabla^2 \Psi=\nabla^2 \Phi$. Hence, for any $x \in B_r(x_0)$, it holds
	\begin{equation*}
		F_p(\nabla \Psi(x), \nabla^2 \Psi(x))< 0
	\end{equation*}
	that is to say
	\begin{equation*}
		-\pL \Psi(x)<0.
	\end{equation*}
	This leads, for any non-negative test function $\varphi \in W^{1,p}_0(B_r(x_0))$ with $\varphi \not \equiv 0$,  to
	\begin{equation*}
		\sum_{j=1}^{N}\int_{B_r(x_0)}|\Psi_{x_j}|^{p-2}\Psi_{x_j}\varphi_{x_j}dx < 0.
	\end{equation*}
	Moreover, being $u$ a weak solution of \eqref{Stkprob}, we have, for any $\varphi \in W^{1,p}_0(B_r(x_0))$,
	\begin{equation*}
		\sum_{j=1}^{N}\int_{B_r(x_0)} |u_{x_j}|^{p-2}u_{x_j}\varphi_{x_j}dx=0.
	\end{equation*}
	Thus we get, for any non-negative text function $\varphi \in W^{1,p}_0(B_r(x_0))$,
	\begin{equation*}
		\sum_{j=1}^{N}\int_{B_r(x_0)} (|\Psi_{x_j}|^{p-2}\Psi_{x_j}-|u_{x_j}|^{p-2}u_{x_j})\varphi_{x_j}dx < 0.
	\end{equation*}
	Let us observe that $\Psi(x_0)-u(x_0)=\frac{m}{2}>0$. On the other hand, since $u$ and $\Phi$ are continuous, there exists a radius $r_*>0$ such that $u(x)-\Phi(x)\ge \frac{m}{2}$ for any $x \in B_r(x_0)\setminus B_{r_*}(x_0)$. In particular, for any $x \in B_r(x_0)\setminus B_{r_*}(x_0)$ it holds $\Psi(x)-u(x)\le 0$. Thus, the function $\varphi=(\Psi-u)^+\chi_{B_r(x_0)}$ can be expressed as
	\begin{equation*}
		\varphi(x)=\begin{cases} (\Psi-u)^+ & x \in B_r(x_0)\\
			0 & x \not \in B_{r_*}(x_0),
		\end{cases}
	\end{equation*} 
	where the two definitions agree in $B_{r}(x_0)\setminus B_{r_*}(x_0)$. Finally, we can observe that, being $\Psi$ and $u$ both in $W^{1,p}(B_r(x_0))$, $\Psi-u$ is in $W^{1,p}(B_r(x_0))$ and then also its positive part (see \cite{savare1996regularity} and references therein). Since we have shown that $\varphi \in W_0^{1,p}(B_r(x_0))$, we can use it as a test function to achieve
	\begin{equation*}
		\sum_{j=1}^{N}\int_{\{\Psi>u\}\cap B_r(x_0)} (|\Psi_{x_j}|^{p-2}\Psi_{x_j}-|u_{x_j}|^{p-2}u_{x_j})(\Psi_{x_j}-u_{x_j})dx<0.
	\end{equation*}
	Thus, by Lemma \ref{lemp}, we obtain
	\begin{align*}
		0\le \sum_{j=1}^{N}&\int_{\{\Psi>u\}\cap B_r(x_0)}|\Psi_{x_j}-u_{x_j}|^pdx \\&\le C(p)\sum_{j=1}^{N}\int_{\{\Psi>u\}\cap B_r(x_0)}(|\Psi_{x_j}|^{p-2}\Psi_{x_j}-|u_{x_j}|^{p-2}u_{x_j})(\Psi_{x_j}-u_{x_j})dx<0, 
	\end{align*}
	which is absurd.\\
	Now let us consider $x_0 \in \partial \Omega$. As before, let us argue by contradiction, supposing that
	\begin{equation*}
		\max\{F_p(\nabla \Phi(x_0),\nabla^2\Phi(x_0)),B_p(\sigma,x_0,u(x_0),\nabla \Phi(x_0))\}<0.
	\end{equation*}
	Thus, since $\Phi \in C^2$ and $u \in C^0$, there exists a radius $r>0$ such that, for any $x \in B_r(x_0)\cap \Omega$, it holds
	\begin{equation*}
		F_p(\nabla \Phi(x),\nabla^2\Phi(x))<0, 
	\end{equation*}
	while, for any $x \in B_r(x_0)\cap \partial \Omega$, it holds
	\begin{equation*}
		\max\{F_p(\nabla \Phi(x),\nabla^2\Phi(x)),B_p(\sigma,x,u(x),\nabla \Phi(x))\}<0.
	\end{equation*}
	As before, let us consider 
	\begin{equation*}
		m=\inf_{x \in \partial B_r(x_0)\cap \overline{\Omega}}|u(x)-\Phi(x)|=\inf_{x \in \partial B_r(x_0)\cap \overline{\Omega}}(u(x)-\Phi(x))
	\end{equation*}
	and define $\Psi(x)=\Phi(x)+\frac{m}{2}$. We have that, for any $x \in B_r(x_0) \cap \Omega$, it holds
	\begin{equation*}
		F_p(\nabla \Psi(x),\nabla^2\Psi(x))<0,
	\end{equation*}
	while, for any $x \in B_r(x_0)\cap \partial \Omega$, it holds
	\begin{equation*}
		\max\{F_p(\nabla \Psi(x),\nabla^2\Psi(x)),B_p(\sigma,x,u(x),\nabla \Psi(x))\}<0.
	\end{equation*}
	From the fact that $F_p(\nabla \Psi(x),\nabla^2 \Psi(x))<0$, we achieve
	\begin{equation*}
		-\pL \Psi(x)<0.
	\end{equation*}
	Let us consider a non-negative test function $\varphi \in W^{1,p}(B_r(x_0) \cap \Omega)$ such that $\varphi \not \equiv 0$ and $\varphi=0$ on $\partial B_r(x_0) \cap \Omega$. It holds
	\begin{equation*}
		\sum_{j=1}^{N}\int_{B_r(x_0) \cap \Omega}|\Psi_{x_j}|^{p-2}\Psi_{x_j}\varphi_{x_j}dx< \sum_{j=1}^{N}\int_{B_r(x_0) \cap \partial \Omega}|\Psi_{x_j}|^{p-2}\Psi_{x_j}\varphi \nu_{\partial \Omega}^jd\cH^{N-1}.
	\end{equation*}
	Now, since $B_p(\sigma,x,u(x),\nabla \Psi(x))<0$, we have, for $x \in B_r(x_0)\cap \partial \Omega$, 
	\begin{equation*}
		\sum_{j=1}^{N}|\Psi_{x_j}(x)|^{p-2}\Psi_{x_j}(x)\nu_{\partial \Omega}^j(x)<\sigma |u(x)|^{p-2}u(x) \rho_p(x)
	\end{equation*}
	and consequently
	\begin{equation*}
		\sum_{j=1}^{N}\int_{B_r(x_0) \cap \Omega}|\Psi_{x_j}|^{p-2}\Psi_{x_j}\varphi_{x_j}dx< \sigma\int_{B_r(x_0) \cap \partial \Omega}|u|^{p-2}u\varphi \rho_pd\cH^{N-1}.
	\end{equation*}
	Moreover, being $u$ a weak solution of \eqref{Stkprob},  we have
	\begin{equation*}
		\sum_{j=1}^{N}\int_{B_r(x_0) \cap \Omega}|u_{x_j}|^{p-2}u_{x_j}\varphi_{x_j}dx=\sigma\int_{B_r(x_0) \cap \partial \Omega}|u|^{p-2}u\varphi. \rho_pd\cH^{N-1}.
	\end{equation*}
	Hence we obtain
	\begin{equation*}
		\sum_{j=1}^{N}\int_{B_r(x_0) \cap \Omega}(|\Psi_{x_j}|^{p-2}\Psi_{x_j}-|u_{x_j}|^{p-2}u_{x_j})\varphi_{x_j}dx<0.
	\end{equation*}
	Let us consider $\varphi=(\Psi-u)^+\chi_{B_r(x_0)\cap \overline{\Omega}}$. Arguing as before we have that $\varphi \in W^{1,p}(\Omega \cap B_r(x_0))$ and $\varphi=0$ on $\partial B_r(x_0) \cap \Omega$, thus we can use it as test function to achieve
	\begin{align*}
		0\le \sum_{j=1}^{N}&\int_{\{\Psi>u\}\cap B_r(x_0)\cap \Omega}|\Psi_{x_j}-u_{x_j}|^pdx\\
		&\le C(p) \sum_{j=1}^{N}\int_{\Omega}(|\Psi_{x_j}|^{p-2}\Psi_{x_j}-|u_{x_j}|^{p-2}u_{x_j})(\Psi_{x_j}-u_{x_j})dx<0, 
	\end{align*}
	which is absurd.
\end{proof}
\begin{rmk}
	Concerning the regularity of a weak solution $u$ of $-\pL u=0$, let us observe that, for $p \ge 2$, orthotropic $p$-harmonic functions are locally Lipschitz in $\Omega$ (see \cite{bousquet2018lipschitz}) and, in particular, in dimension $2$ they are $C^1(\Omega)$ for any $p>1$ (see \cite{bousquet2018c1,lindqvist2018regularity}). We will actually work with $p \to +\infty$, hence we can suppose $p>N$. In such case, Morrey's embedding theorem ensures that $u \in C^0(\overline{\Omega})$. We can conclude that for $p>N$, every weak solution of \eqref{Stkprob} is a viscosity solution of \eqref{Stekrev}.
\end{rmk}
\section{The orthotropic $\infty$-Laplacian: heuristic derivation}
We want to study the problem \eqref{Stkprob} as $p \to +\infty$. To do this, we need to introduce the orthotropic $\infty$-Laplacian as the formal limit as $p \to +\infty$ of $\pL$. The operator  $\pL$ can be interpreted as the anistropic $p$-Laplace operator associated to the norm $\cF_p(x)=\Norm{x}{\ell^p}$, i. e.
\begin{equation*}
	\pL u=\divg\left(\frac{1}{p}\nabla_x \cF_p^p(\nabla u)\right).
\end{equation*}
In the classic case the $\infty$-Laplacian $\ciL$ was achieved from the $p$-Laplacian $\cpL$ by dividing by $(p-2)|\nabla u|^{p-4}$ and then formally taking the limit as $p \to +\infty$ (see \cite{lindqvist2016notes}). Here we work in the same fashion using $\Norm{\nabla u}{\ell^p}$. Before doing this, let us recall the following easy result.
\begin{lem}\label{lemunifconv}
	The functions $\Norm{\cdot}{\ell^p}$ uniformly converge to $\Norm{\cdot}{\ell^\infty}$ as $p \to +\infty$ and to $\Norm{\cdot}{\ell^1}$ as $p \to 1$ in any compact set $K\subseteq \R^N$.
\end{lem} 
\begin{proof}
	Let us recall that for any $x \in \R^N$
	\begin{equation}\label{ineq1}
		\Norm{x}{\ell^\infty}\le \Norm{x}{\ell^p}\le N^{\frac{1}{p}}\Norm{x}{\ell^\infty}, 
	\end{equation}
	thus we have that for any compact $K \subseteq \R^N$ (setting $M_{\infty}=\max_{x \in K}\Norm{x}{\ell^\infty}$)
	\begin{equation*}
		|\Norm{x}{\ell^p}-\Norm{x}{\ell^\infty}|\le (1-N^{\frac{1}{p}})\Norm{x}{\ell^\infty}\le M_\infty(1-N^{\frac{1}{p}}).
	\end{equation*}
	Let us also recall that
	\begin{equation*}
		\Norm{x}{\ell^p}\le \Norm{x}{\ell^1}\le N^{1-\frac{1}{p}}\Norm{x}{\ell^p}, 
	\end{equation*}
	so  we have that for any compact $K \subseteq \R^N$ (setting $M_{1}=\max_{x \in K}\Norm{x}{\ell^1}$)
	\begin{equation*}
		|\Norm{x}{\ell^p}-\Norm{x}{\ell^1}|\le (1-N^{\frac{1}{p}-1})\Norm{x}{\ell^1}\le M_1(1-N^{\frac{1}{p}-1}).
	\end{equation*}
\end{proof}
The previous Lemma allows us to work directly with $\Norm{\nabla u}{\ell^\infty}$, instead of working with $\Norm{\nabla u}{\ell^p}$.
Suppose $u \in C^2$ and write
\begin{align*}
	\pL u&=(p-1)\sum_{j=1}^{N}|u_{x_j}|^{p-4}u_{x_j}^2u_{x_j,x_j},
\end{align*}
i.e.
\begin{equation*}
	\frac{\pL u}{p-1}=\sum_{j=1}^{N}|u_{x_j}|^{p-4}u_{x_j}^2u_{x_j,x_j}.
\end{equation*}
Dividing everything by $\Norm{\nabla u}{\ell^\infty}^{p-4}$, we achieve
\begin{equation}\label{pLdiv}
	\frac{\pL u}{(p-1)\Norm{\nabla u}{\ell^\infty}^{p-4}}=\sum_{j=1}^{N}\left|\frac{u_{x_j}}{\Norm{\nabla u}{\ell^\infty}}\right|^{p-4}u_{x_j}^2u_{x_j,x_j}.
\end{equation}
If we  consider the set
\begin{equation*}
	I(x):=\{j \le N: \ |x_j|=\Norm{x}{\ell^\infty}\},
\end{equation*}
we can rewrite equation \eqref{pLdiv} as
\begin{equation*}\label{forma_comoda}
	\frac{\pL u}{(p-1)\Norm{\nabla u}{\ell^\infty}^{p-4}}=\sum_{j \in I(\nabla u(x))}u_{x_j}^2u_{x_j,x_j}+\sum_{j \not \in I(\nabla u(x))}\left|\frac{u_{x_j}}{\Norm{\nabla u}{\ell^\infty}}\right|^{p-4}u_{x_j}^2u_{x_j,x_j}.
\end{equation*}
Finally, taking the limit as $p \to +\infty$ and recalling that for any $j \not \in I_\infty(\nabla u(x))$ we have $\left|\frac{u_{x_j}}{\Norm{\nabla u}{\ell^\infty}}\right|<1$, we achieve
\begin{equation*}
	\iL u=\lim_{p \to +\infty}\frac{\pL u}{(p-1)\Norm{\nabla u}{\ell^\infty}^{p-4}}=\sum_{j \in I(\nabla u(x))}u_{x_j}^2u_{x_j,x_j}=\Norm{\nabla u}{\ell^\infty}^2\sum_{j \in I(\nabla u(x))}u_{x_j,x_j}.
\end{equation*}
The same result holds also if we use $\Norm{\nabla u}{\ell^p}$ in place of $\Norm{\nabla u}{\ell^\infty}$, since, by uniform convergence, for $p$ big enough and $j \not \in I(\nabla u(x))$,  we still have $\left|\frac{u_{x_j}}{\Norm{\nabla u}{\ell^p}}\right|<1$.\\
We stress the fact  that the computations above are just heuristics, whose aim is  to obtain an expected form of the limit operator; it  turns out that such heuristics actually lead to the limit operator of the orthotropic $p$-Laplacian. Indeed, the orthotropic $\infty$-Laplacian has been introduced in \cite{belloni2004pseudo} as
\begin{equation*}
	\iL u= \sum_{j \in I(\nabla u(x))}u_{x_j}^2u_{x_j,x_j}.
\end{equation*}
In the same paper the authors prove that this operator is related to the problem of the Absolutely Minimizing Lipschitz Extension with respect to the $\ell^\infty$ on $\R^N$ (as the $\infty$-Laplacian is related to the same problem with respect to the $\ell^2$ norm, as shown in  \cite{aronsson1967extension}). In particular, in \cite{belloni2004pseudo} it is shown that, if $u \in C^2(\Omega) \cap W^{1, \infty}(\Omega)$ is such that for any $D \subset \subset \Omega$ and any $w \in u+W^{1,\infty}_0(\Omega)$ it holds 
\begin{equation*}
	\Norm{\Norm{\nabla u}{\ell^\infty}}{L^\infty(D)}\le \Norm{\Norm{\nabla w}{\ell^\infty}}{L^\infty(D)},
\end{equation*}
then $u$ solves
\begin{equation*}
	-\iL u=0.
\end{equation*}
In the following we will work with a limit problem arising from \eqref{Stkprob} as $p \to +\infty$ that will take into account the operator $\iL$.
\section{Limit eigenvalues}
In this section we study the  behaviour of  the first non-trivial Steklov eigenvalue  as $p \to +\infty$. As we stated in Section \ref{np}, for any $p>1$, we have $\sigma_{1,p}(\Omega)=0$, thus  $\lim_{p \to +\infty}\sigma_{1,p}(\Omega)=0$.  For this reason we focus on $\Sigma_{p}(\Omega)$.\\
In order to  determine $\lim_{p \to +\infty} \Sigma_p(\Omega)$, we first need to fix some notations. For any measurable set $E \subseteq \R^N$, we denote by $V(E)$ the Lebesgue measure of $E$, by $\cH^{N-1}(E)$ the $(N-1)$- dimensional Hausdorff measure of $E$ and we set 
\begin{equation*}
	d_1(x,y)=\Norm{x-y}{\ell^1}, \ x,y \in \R^N.
\end{equation*}
For fixed $x_0$, the function $x \mapsto d_1(x,x_0)$ is such that $\Norm{\nabla d_1(x,x_0)}{\ell^\infty}=1$ almost everywhere, as observed in \cite{chambolle2019existence}. Moreover, let us define the quantity
\begin{equation*}
	\diam_1(E)=\sup_{x,y\in E}d_1(x,y).
\end{equation*}
Now let us recall the variational characterization of $\Sigma_p^p(\Omega)$ given in equation \eqref{varcar} and let us use the following notation:
\begin{align*}
	\cR_p[u]&=\frac{\int_{\Omega}\Norm{\nabla u}{\ell^p}^pdx}{\int_{\partial \Omega}|u|^p\rho_pd\cH^{N-1} },\\
	\cM_p[u]&=\int_{\partial \Omega}|u|^{p-2}u\rho_pd\cH^{N-1} ,\\
	\cU_p&=\left\{u \in W^{1,p}(\Omega): \ \cM_p[u]=0\right\},
\end{align*}
so that we can rewrite $\Sigma_p^p(\Omega)=\min_{u \in \cU_p}\cR_p[u]$. We consider on $L^p(\Omega)$ the norm
\begin{equation*}
	\Norm{u}{L^p(\Omega)}^p=\fint_{\Omega}|u|^pdx=\dfrac{1}{V(\Omega)}\int_{\Omega}|u|^pdx .
\end{equation*}
On $\partial \Omega$, we define the measure $d\cH_p=\rho_p  d\cH^{N-1}$ and  consider for any $p,q \ge 1$
\begin{equation*}
	\Norm{u}{L^p(\partial\Omega, \cH_q)}^p=\frac{1}{\cH^{N-1}(\partial \Omega)}\int_{\partial \Omega} |u|^pd\cH_q.
\end{equation*}
Recall that if $q=2$, then $\rho_2 \equiv 1$ and $\Norm{u}{L^p(\partial \Omega, \cH_2)}=\Norm{u}{L^p(\partial \Omega)}$. From the equivalence of the $\ell^p$ norms on $\R^N$, that, for $p>q\geq 1$, is given by
\begin{equation*}
	\Norm{x}{\ell^p} \le \Norm{x}{\ell^q}\le N^{\frac{1}{q}-\frac{1}{p}}\Norm{x}{\ell^p},
\end{equation*}
we have that, for $p>q\geq 1$,
\begin{equation*}
	\rho_q(x) \le \rho_p(x)\le N^{\frac{1}{p'}-\frac{1}{q'}}\rho_q(x).
\end{equation*}
Moreover, we have,  from Lemma \ref{lemunifconv} and this equivalence, the following result.
\begin{lem}\label{lem41}
	For any $p,q_1,q_2 \ge 1$ we have $u \in L^p(\partial \Omega, \cH_{q_1})$ if and only if $u \in L^p(\partial \Omega, \cH_{q_2})$ and, if $1 \le q_1<q_2 \le +\infty$,
	\begin{equation*}
		\Norm{u}{L^p(\partial \Omega, \cH_{q_1})} \le \Norm{u}{L^p(\partial \Omega, \cH_{q_2})}\le N^{\frac{1}{q_2'}-\frac{1}{q_1'}}\Norm{u}{L^p(\partial \Omega, \cH_{q_1})}(x).
	\end{equation*}
	Moreover, as $q \to +\infty$, we have $\Norm{u}{L^p(\partial \Omega, \cH_q)} \to \Norm{u}{L^p(\partial \Omega, \cH_\infty)}$ and, as $q \to 1$, we have $\Norm{u}{L^p(\partial \Omega, \cH_q)} \to \Norm{u}{L^p(\partial \Omega, \cH_1)}$.
\end{lem}
The latter property is due to the fact that, since $\rho_p(x)=\Norm{\nu_{\partial \Omega}(x)}{\ell^{p'}}$ and $\nu_{\partial \Omega}(x) \in \mathbb{S}^{N-1}$, where $\mathbb{S}^{N-1}$ is the unit sphere of $\R^N$ with respect to the $\ell^2$ norm (that is a compact set), $\rho_p(x) \to \rho_\infty(x)$ uniformly as $p \to +\infty$ and $\rho_p(x) \to \rho_1(x)$ uniformly as $p \to 1$.\\
Let us observe that we can recast $\cR_p[u]$ as
\begin{equation*}
	\cR_p[u]=\frac{\fint_{\Omega}\Norm{\nabla u}{\ell^p}^pdx}{\frac{1}{V(\Omega)}\int_{\partial \Omega}|u|^pd\cH_p}.
\end{equation*}
Moreover, we have the following lower-semicontinuity property.
\begin{lem} \label{lemlim}
	Fix $p\ge 2$ and let $u_n \rightharpoonup u$ in $W^{1,p}(\Omega)$. Then
	\begin{equation*}
		\cR_p[u]\le \liminf_{n \to +\infty}\cR_p[u_n].
	\end{equation*}
\end{lem}
Let us denote with $u_{2,p} \in \cU_p$ a minimizer of $\cR_p$ such that
\begin{equation}\label{eq:normcond}
	\frac{1}{V(\Omega)}\int_{\partial \Omega}|u_{2,p}|^pd\cH_p=1.
\end{equation}
In particular, in such a case,
\begin{equation}\label{eq:varcar2}
	\Sigma_p^p(\Omega)=\fint_{\Omega}\Norm{\nabla u_{2,p}}{\ell^p}^pdx.
\end{equation}
We need the following technical Lemma.
\begin{lem}
	Let $\Omega$ be a bounded open convex subset of $\R^N$ and $u \in W^{1,\infty}(\Omega)$. Then
	\begin{equation}\label{ineqLip}
		|u(x)-u(y)|\le \Norm{\Norm{\nabla u}{\ell^\infty}}{L^\infty}\diam_1(\Omega), \ \forall x,y \in \overline{\Omega}.
	\end{equation}
\end{lem}
\begin{proof}
	Let us recall that, by definition of polar norm, $|\langle x,y \rangle|\le \Norm{x}{\ell^\infty}\Norm{y}{\ell^1}$. Now fix $x,y \in \Omega$ and observe that,  since $\Omega$ is convex,  $(1-t)x+ty \in \Omega$ for any $t \in [0,1]$. Define the function
	\begin{equation*}
		v(t)=u((1-t)x+ty), \ t \in [0,1]
	\end{equation*} 
	and observe that $v \in W^{1, \infty}([0,1])$. Hence, in particular, $v$ is absolutely continuous and 
	\begin{equation*}
		v(0)-v(1)=\int_0^1 v'(t)dt, 
	\end{equation*}
	where $v'$ is the weak derivative. We have
	\begin{equation*}
		u(x)-u(y)=\int_0^1 \langle \nabla u ((1-t)x+ty),y-x\rangle dt
	\end{equation*}
	and then
	\begin{align*}
		|u(x)-u(y)|&\le \int_0^1 |\langle \nabla u ((1-t)x+ty),y-x\rangle| dt\\&\le \int_0^1 \Norm{\nabla u ((1-t)x+ty)}{\ell^\infty}\Norm{x-y}{\ell^1} dt\\&\le \Norm{\Norm{\nabla u}{\ell^\infty}}{L^\infty(\Omega)}\diam_1(\Omega).
	\end{align*}
	Finally, by Morrey's embedding theorem, we know that $u \in C^0(\overline{\Omega})$, thus inequality \eqref{ineqLip} holds also for $x,y \in \partial \Omega$. 
\end{proof}
Now we show the following result.
\begin{prop}\label{prop_lim}
	It holds
	\begin{equation*}
		\lim_{p \to +\infty}\Sigma_p(\Omega)=\frac{2}{\diam_1(\Omega)}=:\Sigma_\infty(\Omega).
	\end{equation*}
\end{prop}
\begin{proof}
	First of all, let us show that $\limsup_{p \to +\infty}\Sigma_p \le \frac{2}{\diam_1(\Omega)}$. To do this, we consider $x_0 \in \Omega$ and we observe that, being $\Omega$ an open set, $d_1(x,x_0)>0$ for any $x \in \partial \Omega$. Indeed, if $d_1(x,x_0)=0$ for some $x \in \partial \Omega$, being $d_1$ a distance, we should have $x=x_0$ and then $x_0 \in \Omega \cap \partial \Omega= \emptyset$. In particular, this implies that $\mathcal{M}_p[d_1(\cdot,x_0)]>0$.\\
	Define the function $w_p(x)=d_1(x,x_0)-c_p$ where $c_p \in \R$ is chosen in such a way that $w_p\in \cU_p$. Let us recall that $\Norm{\nabla w_p}{\ell^\infty}=1$ almost everywhere in $\Omega \setminus\{x_0\}$, hence we have, by equation \eqref{ineq1},
	\begin{equation*}
		\fint_{\Omega}\Norm{\nabla w_p}{\ell^p}^pdx\le N.
	\end{equation*}
	Moreover, we have, from Lemma \ref{lem41},
	\begin{equation*}
		\Norm{w_p}{L^p(\partial \Omega, \cH_\infty)}\le N^{\frac{1}{p}}\Norm{w_p}{L^p(\partial \Omega, \cH_p)}.
	\end{equation*}
	Thus, recalling that $\Sigma_p(\Omega)\le \cR[w_p]^{\frac{1}{p}}$, we achieve
	\begin{align}\label{Sigmap1}
		\begin{split}
			\Sigma_p(\Omega)&\le \frac{\left(\fint_{\Omega}\Norm{\nabla w_p}{\ell^p}^pdx\right)^{\frac{1}{p}}}{\left(\frac{1}{V(\Omega)}\int_{\partial \Omega}|w_p|^p\rho_{p}(x)d\cH^{N-1}\right)^{\frac{1}{p}}}\\&=\frac{\left(\fint_{\Omega}\Norm{\nabla w_p}{\ell^p}^pdx\right)^{\frac{1}{p}}}{\left(\frac{\cH^{N-1}(\partial \Omega)}{V(\Omega)}\right)^{\frac{1}{p}}\Norm{w_p}{L^p(\partial \Omega, \cH_p)}}\\
			&\le \frac{N^{\frac{1}{p}}}{\left(\frac{\cH^{N-1}(\partial \Omega)}{V(\Omega)}\right)^{\frac{1}{p}}N^{-\frac{1}{p}}\Norm{w_p}{L^p(\partial \Omega, \cH_\infty)}}.
		\end{split}
	\end{align}
	Now let us observe that, since $\cM_p(w_p)=0$, $w_p$ must change sign on $\partial \Omega$. Since $0 \le d_1(x,x_0)\le \diam_1(\Omega)$, we have $c_p \in [0,\diam_1(\Omega)]$. Up to a subsequence, we can suppose $c_p \to c \in [0,\diam_1(\Omega)]$ as $p \to +\infty$ and, setting $w=d_1(x,x_0)-c$,  we have that $w_p \to w$ uniformly. Hence, as $p \to +\infty$, 
	\begin{equation*}
		\Norm{w_p}{L^p(\partial \Omega, \cH_\infty)} \to \sup_{x \in \partial \Omega}|d_1(x,x_0)-c|;
	\end{equation*}
	taking the $\limsup$ as $p\to+\infty$ in \eqref{Sigmap1},  we have
	\begin{equation}\label{eq:limsuppass1}
		\limsup_{p \to +\infty}\Sigma_p(\Omega)\le \frac{1}{\sup_{x \in \partial \Omega}|d_1(x,x_0)-c|}.
	\end{equation}
	Now let us observe that 
	\begin{equation*}
		|d_1(x,x_0)-c|\ge \inf_{c \in [0,\diam_1(\Omega)]}|d_1(x,x_0)-c|=\frac{d_1(x,x_0)}{2},
	\end{equation*}
	thus we get 
	\begin{equation*}
		\sup_{x \in \partial \Omega}|d_1(x,x_0)-c|\ge \frac{\sup_{x \in \partial \Omega}d_1(x,x_0)}{2}.
	\end{equation*}
	Plugging this relation into equation \eqref{eq:limsuppass1},  we achieve
	\begin{equation*}
		\limsup_{p \to +\infty}\Sigma_p(\Omega) \le\frac{2}{\sup_{x \in \partial \Omega}d_1(x,x_0)}.
	\end{equation*}
	Since this inequality holds for any $x_0 \in \Omega$, we can take the infimum as $x_0 \in \Omega$ to  obtain	\begin{equation*}
		\limsup_{p \to +\infty}\Sigma_p(\Omega) \le\frac{2}{\diam_1(\Omega)}.
	\end{equation*}
	Now let us show that $\liminf_{p \to +\infty}\Sigma_p(\Omega) \ge \frac{2}{\diam_1(\Omega)}$. To do this, let us consider $m>N$ and $p>m$. Since $p>2$,  we have
	\begin{equation*}
		\Norm{\nabla u_{2,p}}{\ell^p} \ge N^{\frac{1}{p}-2}\Norm{\nabla u_{2,p}}{\ell^2}, 
	\end{equation*}
	and then, by H\"older inequality  and definition of $u_{2,p}$, 
	\begin{equation*}
		\Sigma_p(\Omega)=\left(\fint_{\Omega}\Norm{\nabla u_{2,p}}{\ell^p}^pdx\right)^{\frac{1}{p}} \ge N^{\frac{1}{p}-2}\left(\fint_{\Omega}\Norm{\nabla u_{2,p}}{\ell^2}^pdx\right)^{\frac{1}{p}}\ge N^{\frac{1}{p}-2}\left(\fint_{\Omega}\Norm{\nabla u_{2,p}}{\ell^2}^mdx\right)^{\frac{1}{m}}.	
	\end{equation*}
	Since $W^{1,m}(\Omega)$ is compactly embedded in $C^0(\overline{\Omega})$, we can suppose (up to a subsequence) that there exists a function $u_{2,\infty} \in C^0(\overline{\Omega})$ such that $u_{2,p} \to u_{2, \infty}$ uniformly on $\overline{\Omega}$ and weakly in $W^{1,m}(\Omega)$. Now let us fix any $1 \le q <p$ and observe that, by lower semicontinuity of the functional
	\begin{equation*}
		u \in W^{1,q}(\Omega) \mapsto (\cH^{N-1}(\partial\Omega))^{\frac{1}{q}}\cR_q[u] \in \R
	\end{equation*}
	with respect to the weak convergence in $W^{1,q}$ as stated in Lemma \ref{lemlim}, we have
	\begin{equation*}
		\frac{\left(\fint_\Omega \Norm{\nabla u_{2,\infty}}{\ell^\infty}^qdx\right)^\frac{1}{q}}{\left(\frac{1}{\cH^{N-1}(\partial\Omega)}\int_{\partial\Omega} |u_{2,\infty}|^q\rho_{\infty}d\cH^{N-1}\right)^\frac{1}{q}}
		\le \liminf_{p \to +\infty}\frac{\left(\fint_\Omega \Norm{\nabla u_{2,p}}{\ell^\infty}^qdx\right)^\frac{1}{q}}{\left(\frac{1}{\cH^{N-1}(\partial\Omega)}\int_{\partial\Omega} |u_{2,p}|^q\rho_{\infty}d\cH^{N-1}\right)^\frac{1}{q}}.
	\end{equation*}
	By H\"older inequality we get 
	\begin{equation*}
		\frac{\left(\fint_\Omega \Norm{\nabla u_{2,\infty}}{\ell^\infty}^qdx\right)^\frac{1}{q}}{\left(\frac{1}{\cH^{N-1}(\partial\Omega)}\int_{\partial\Omega} |u_{2,\infty}|^q\rho_{\infty}d\cH^{N-1}\right)^\frac{1}{q}}
		\le\liminf_{p \to +\infty}\frac{\left(\fint_\Omega \Norm{\nabla u_{2,p}}{\ell^\infty}^pdx\right)^\frac{1}{p}}{\left(\frac{1}{\cH^{N-1}(\partial\Omega)}\int_{\partial\Omega} |u_{2,p}|^q\rho_\infty d\cH^{N-1}\right)^\frac{1}{q}}
	\end{equation*}
	and then, by using \eqref{ineq1},
	\begin{equation*}
		\frac{\left(\fint_\Omega \Norm{\nabla u_{2,\infty}}{\ell^\infty}^qdx\right)^\frac{1}{q}}{\left(\frac{1}{\cH^{N-1}(\partial\Omega)}\int_{\partial\Omega} |u_{2,\infty}|^q\rho_{\infty}d\cH^{N-1}\right)^\frac{1}{q}}
		\le\liminf_{p \to +\infty}\frac{\left(\fint_\Omega \Norm{\nabla u_{2,p}}{\ell^p}^pdx\right)^\frac{1}{p}}{\left(\frac{1}{\cH^{N-1}(\partial\Omega)}\int_{\partial\Omega} |u_{2,p}|^q\rho_\infty d\cH^{N-1}\right)^\frac{1}{q}}.
	\end{equation*}
	Recalling equations \eqref{eq:normcond} and \eqref{eq:varcar2}, we have
	\begin{align*}
		\frac{\left(\fint_\Omega \Norm{\nabla u_{2,\infty}}{\ell^\infty}^qdx\right)^\frac{1}{q}}{\left(\frac{1}{\cH^{N-1}(\partial\Omega)}\int_{\partial\Omega} |u_{2,\infty}|^q\rho_{\infty}d\cH^{N-1}\right)^\frac{1}{q}}
		&\le\liminf_{p \to +\infty}\frac{\left(\frac{1}{V(\Omega)}\int_{\partial\Omega} |u_{2,p}|^p\rho_pd\cH^{N-1}\right)^\frac{1}{p}}{\left(\frac{1}{\cH^{N-1}(\partial\Omega)}\int_{\partial\Omega} |u_{2,p}|^q\rho_\infty d\cH^{N-1}\right)^\frac{1}{q}}\Sigma_p(\Omega)\\
		&=\liminf_{p \to +\infty}\frac{\left(\frac{\cH^{N-1}(\partial\Omega)}{V( \Omega)}\right)^{\frac{1}{p}}\Norm{u_{2,p}}{L^p(\partial \Omega, \cH_p)}}{\Norm{u_{2,p}}{L^q(\partial \Omega, \cH_\infty)}}\Sigma_p(\Omega)
	\end{align*}
	and, by using Lemma \ref{lem41}, we achieve
	\begin{align*}
		\frac{\left(\fint_\Omega \Norm{\nabla u_{2,\infty}}{\ell^\infty}^qdx\right)^\frac{1}{q}}{\left(\frac{1}{\cH^{N-1}(\partial\Omega)}\int_{\partial\Omega} |u_{2,\infty}|^q\rho_{\infty}d\cH^{N-1}\right)^\frac{1}{q}}&\le \liminf_{p \to +\infty}\frac{\left(\frac{\cH^{N-1}(\partial\Omega)}{V( \Omega)}\right)^{\frac{1}{p}}\Norm{u_{2,p}}{L^p(\partial \Omega, \cH_\infty)}}{\Norm{u_{2,p}}{L^q(\partial \Omega, \cH_\infty)}}\Sigma_p(\Omega)\\&=\frac{\Norm{u_{2,\infty}}{L^\infty(\partial \Omega)}}{\Norm{u_{2,\infty}}{L^q(\partial \Omega,\cH_\infty)}}\liminf_{p \to +\infty}\Sigma_p(\Omega);
	\end{align*}
	finally  let us take the limit as $q \to +\infty$ to obtain
	\begin{equation}\label{eq:pass}
		\frac{\Norm{\Norm{\nabla u_{2,\infty}}{\ell^\infty}}{L^\infty(\Omega)}}{\Norm{u_{2,\infty}}{L^\infty(\partial \Omega)}}\le \liminf_{p \to +\infty}\Sigma_p(\Omega).
	\end{equation}
	Now we want to estimate the left-hand side of the previous inequality. To do this, let us recall that, for any $p>1$,
	\begin{equation*}
		\int_{\partial \Omega}|u_{2,p}|^{p-2}u_{2,p}\rho_p d\cH^{N-1}=0
	\end{equation*}
	hence, in particular,
	\begin{equation*}
		\Norm{(u_{2,p})_+}{L^{p-1}(\partial \Omega, \cH_p)}=\Norm{(u_{2,p})_-}{L^{p-1}(\partial \Omega, \cH_p)}.
	\end{equation*}
	By using the previous identity  we have
	\begin{align}\label{pass1}
		\begin{split}
			0 &\le \left|\Norm{(u_{2,\infty})_+}{L^{p-1}(\partial \Omega, \cH_p)}-\Norm{(u_{2,\infty})_-}{L^{p-1}(\partial \Omega, \cH_p)}\right|\\
			&\le \left|\Norm{(u_{2,\infty})_+}{L^{p-1}(\partial \Omega, \cH_p)}-\Norm{(u_{2,p})_+}{L^{p-1}(\partial \Omega, \cH_p)}\right|\\&+\left|\Norm{(u_{2,p})_-}{L^{p-1}(\partial \Omega, \cH_p)}-\Norm{(u_{2,\infty})_-}{L^{p-1}(\partial \Omega, \cH_p)}\right|\\
			&\le \Norm{(u_{2,\infty})_+-(u_{2,p})_+}{L^{p-1}(\partial \Omega, \cH_p)}+\Norm{(u_{2,\infty})_--(u_{2,p})_-}{L^{p-1}(\partial \Omega, \cH_p)}\\
			&\le \Norm{(u_{2,\infty})_+-(u_{2,p})_+}{L^{p-1}(\partial \Omega, \cH_\infty)}+\Norm{(u_{2,\infty})_--(u_{2,p})_-}{L^{p-1}(\partial \Omega, \cH_\infty)}\\
			&\le \left(\frac{\cH_\infty(\partial \Omega)}{\cH^{N-1}(\partial \Omega)}\right)^{\frac{1}{p-1}}\left(\Norm{(u_{2,\infty})_+-(u_{2,p})_+}{L^{\infty}(\partial \Omega)}+\Norm{(u_{2,\infty})_--(u_{2,p})_-}{L^{\infty}(\partial \Omega)}\right).
		\end{split}
	\end{align}
	Now let us observe that
	\begin{equation*}
		N^{-\frac{1}{p}}\Norm{(u_{2,\infty})_\pm}{L^{p-1}(\partial \Omega, \cH_\infty)}\le \Norm{(u_{2,\infty})_{\pm}}{L^{p-1}(\partial \Omega, \cH_p)}\le \Norm{(u_{2,\infty})_\pm}{L^{p-1}(\partial \Omega, \cH_\infty)}
	\end{equation*}
	and $\lim_{p \to +\infty}\Norm{(u_{2,\infty})_\pm}{L^{p-1}(\partial \Omega, \cH_\infty)}=\Norm{(u_{2,\infty})_\pm}{L^{\infty}(\partial \Omega)}$, thus we have
	\begin{equation*}
		\lim_{p \to +\infty}\Norm{(u_{2,\infty})_\pm}{L^{p-1}(\partial \Omega, \cH_p)}=\Norm{(u_{2,\infty})_\pm}{L^{\infty}(\partial \Omega)}.
	\end{equation*}
	Taking the limit as $p \to +\infty$ in \eqref{pass1}, by also using the uniform convergence of $u_{2,p}$ towards $u_{2,\infty}$ on $\partial \Omega$, we obtain
	\begin{equation*}
		0 \le \left|\Norm{(u_{2,\infty})_+}{L^{\infty}(\partial \Omega)}-\Norm{(u_{2,\infty})_-}{L^{\infty}(\partial \Omega)}\right|\le 0;
	\end{equation*}
	thus, being $u_{2,\infty}\in C^0(\overline{\Omega})$,
	\begin{equation*}
		\max_{x \in \partial \Omega}u_{2,\infty}(x)=-\min_{x \in \partial \Omega}u_{2,\infty}(x).
	\end{equation*}
	Let us consider $x_M,x_m \in \partial \Omega$ respectively a maximum and minimum point of $u_{2,\infty}$ on $\partial \Omega$ and observe that $u_{2,\infty}(x_M)=-u_{2,\infty}(x_m)$. This means that $x_M$ and $x_m$ are both maximum points for $|u_{2,\infty}|$ on $\partial \Omega$ and
	\begin{equation*}
		2\Norm{u_{2,\infty}}{L^\infty(\partial \Omega)}=u_{2,\infty}(x_M)-u_{2,\infty}(x_m).
	\end{equation*}
	By \eqref{ineqLip}, we obtain
	\begin{equation*}
		\Norm{u_{2,\infty}}{L^\infty(\partial \Omega)}=\frac{u_{2,\infty}(x_M)-u_{2,\infty}(x_m)}{2}\le \frac{\diam_1(\Omega)}{2}\Norm{\Norm{\nabla u}{\ell^\infty}}{L^\infty(\Omega)}.
	\end{equation*}
	Plugging last inequality in equation \eqref{eq:pass},  we obtain
	\begin{equation*}
		\frac{2}{\diam_1(\Omega)}\le \liminf_{p \to +\infty}\Sigma_p(\Omega), 
	\end{equation*}
	concluding the proof.
\end{proof}
By using the function $u_{2,\infty}$ defined in the previous proof, we can also exploit the behaviour of $\Sigma_\infty(\Omega)$ as a minimizer of a Rayleigh quotient.
\begin{prop}
	It holds
	\begin{equation*}
		\Sigma_\infty(\Omega)=\min\left\{\frac{\Norm{\Norm{\nabla u}{\ell^{\infty}}}{L^\infty(\Omega)}}{\Norm{u}{L^\infty(\partial \Omega)}}, \ u \in W^{1,\infty}(\Omega), \ \max_{x \in \partial \Omega} u(x)=-\min_{x \in \partial \Omega} u(x)\not = 0\right\}.
	\end{equation*}
\end{prop}
\begin{proof}
	Let us consider $u \in W^{1,\infty}(\Omega)$ such that $$u_M:=\max_{x \in \partial \Omega}u(x)=-\min_{x \in \partial \Omega}u(x)=:-u_m.$$ Then, being $\Omega$ an open bounded convex set, we know that $u \in W^{1,p}(\Omega)$ for any $p \ge 1$. Now let us consider $p_n \to +\infty$ as $n \to +\infty$. For each $n \in \N$, let us define $c_n$ such that
	\begin{equation}\label{condp}
		\int_{\partial \Omega}|u+c_n|^{p_n-2}(u+c_n)\rho_{p_n} d\cH^{N-1}=0.
	\end{equation}
	Since $u_M=-u_m$, we know that $u$ changes sign. Moreover, also $u+c_n$ must change sign for any $n \in \N$, hence we have that $c_n \in [-u_M,u_M]$. Let us then consider a subsequence, that we still   call  $c_n$,  such that $c_n \to c \in [-u_M,u_M]$ and let us define $u_n=u+c_n$; it is easy to check that $u_n \to u$ in $C^0(\overline{\Omega})$. By equation \eqref{condp} we have
	\begin{equation*}
		\Norm{(u+c_n)_+}{L^{p_n-1}(\partial \Omega,\cH_{p_n})}= \Norm{(u+c_n)_-}{L^{p_n-1}(\partial \Omega,\cH_{p_n})}
	\end{equation*}
	and then, taking the limit as $n \to +\infty$, by uniform convergence, we have
	\begin{equation*}
		u_M+c=\max_{x \in \partial \Omega}(u+c)=-\min_{x \in \partial \Omega}(u+c)=-u_m-c, 
	\end{equation*}
	and, since $u_M=-u_m$, $c=0$.\\
	Now let us observe that $u_n \in \cU_p$, thus, by definition of $\Sigma_{p_n}(\Omega)$ and recalling that $\nabla u_n=\nabla u$, we achieve,
	\begin{equation*}
		\Sigma_{p_n}(\Omega) \le \frac{|\Omega|^{\frac{1}{p_n}}\left(\fint_{\Omega}\Norm{\nabla u}{\ell^{p_n}}^{p_n}dx\right)^{\frac{1}{p_n}}}{\cH_{p_n}(\partial \Omega)^{\frac{1}{p_n}}\left(\fint_{\partial\Omega}|u_n|^{p_n}\rho_{p_n}(x)d\cH^{N-1}\right)^{\frac{1}{p_n}}}.
	\end{equation*}
	Since $u_n$ converges uniformly towards $u$, we have, by taking the limit as $n \to +\infty$,
	\begin{equation*}
		\Sigma_{\infty}(\Omega) \le \frac{\Norm{\Norm{\nabla u}{\ell^{\infty}}}{L^\infty(\Omega)}}{\Norm{u}{L^\infty(\partial \Omega)}}.
	\end{equation*}
	By arbitrary of $u$, we have
	\begin{equation*}
		\Sigma_{\infty}(\Omega) \le \inf\left\{\frac{\Norm{\Norm{\nabla u}{\ell^{\infty}}}{L^\infty(\Omega)}}{\Norm{u}{L^\infty(\partial \Omega)}}, \ u \in W^{1,\infty}(\Omega), \ \max_{x \in \partial \Omega} u(x)=-\min_{x \in \partial \Omega} u(x)\not = 0\right\}.
	\end{equation*}
	Finally, let us observe that $u_{2,\infty} \in W^{1,\infty}(\Omega)$ and  $\max_{x \in \partial \Omega} u_{2,\infty}(x)=-\min_{x \in \partial \Omega} u_{2,\infty}(x)\not = 0$, hence
	\begin{equation*}
		\frac{\Norm{\Norm{\nabla u_{2,\infty}}{\ell^{\infty}}}{L^\infty(\Omega)}}{\Norm{u_{2,\infty}}{L^\infty(\partial \Omega)}}\ge \Sigma_{\infty}(\Omega).
	\end{equation*}
	However, we also  have 
	\begin{equation*}
		\frac{\Norm{\Norm{\nabla u_{2,\infty}}{\ell^{\infty}}}{L^\infty(\Omega)}}{\Norm{u_{2,\infty}}{L^\infty(\partial \Omega)}}\le \frac{2}{\diam_1(\Omega)}= \Sigma_\infty(\Omega), 
	\end{equation*}
	concluding the proof.
\end{proof}
\begin{rmk}
	Let us observe that, defining $$\mathcal{R}_\infty[u]=\frac{\Norm{\Norm{\nabla u}{\ell^{\infty}}}{L^\infty(\Omega)}}{\Norm{u}{L^\infty(\partial \Omega)}}$$ for $u \in W^{1,\infty}(\Omega)$ with $u \not \equiv 0$ on $\partial \Omega$, the function $u_{2,\infty}$ is a minimizer of $\mathcal{R}_\infty$ in $$\mathcal{U}_\infty=\left\{u \in W^{1,\infty}(\Omega), \ \max_{x \in \partial \Omega}u(x)=-\min_{x \in \partial \Omega}u(x)\not = 0\right\}.$$\\
	Moreover, the previous Proposition also implies that, for any $u \in \mathcal{U}_\infty$, it holds
	\begin{equation*}
		\Norm{u}{L^\infty(\partial \Omega)}\le \frac{1}{\Sigma_\infty(\Omega)}\Norm{\Norm{\nabla u}{\ell^{\infty}}}{L^\infty(\Omega)},
	\end{equation*}
	thus $1/\Sigma_\infty(\Omega)$ represents the best constant of a trace-type inequality in $\mathcal{U}_\infty$.
\end{rmk}
Next step is to characterize $u_{2,\infty}$ as a solution (in the viscosity sense) of a boundary-value problem involving the orthotropic $\infty$-Laplacian as $\Omega$ is regular enough.
\begin{thm}\label{viscosity_solution}
	Let $\Omega$ be an open set with $C^1$ boundary. Then $u_{2,\infty}$ is a viscosity solution of 
	\begin{equation*}
		\begin{cases}
			-\widetilde{\Delta}_{\infty} u_{2,\infty}=0 & \mbox{on } \Omega\\
			\Lambda(x,u,\nabla u)=0 & \mbox{on }\partial \Omega,
		\end{cases}
	\end{equation*}
	where, for $(x,u,\eta) \in \partial \Omega \times \R \times \R^N$,
	\begin{equation*}
		\Lambda(x,u,\eta)=\begin{cases}
			\min\big\{  \Norm{\eta}{\ell^\infty}  -\Sigma_{\infty}(\Omega)|u|\;, \;\sum_{j \in I(\eta)}\eta_{j}\nu^j_{\partial\Omega}(x)  \big \} & \mbox{if } u>0\\
			\max\big\{ \Sigma_{\infty}(\Omega)|u|- \Norm{\eta}{\ell^\infty}, \sum_{j \in I(\eta)}\eta_{j}\nu^j_{\partial\Omega}(x)  \big \} & \mbox{if } u<0\\
			\sum_{j \in I(\eta)}\eta_{j}\nu^j_{\partial\Omega}(x) & \mbox{if } u=0
		\end{cases}
	\end{equation*}
\end{thm}

\begin{proof}
	First of all, we prove  that $-\widetilde{\Delta}_{\infty} u_{2,\infty}=0$  in the viscosity sense in $\Omega$. In order to do that, let us take  a test function  $\Phi$ touching $u$ from above in $x_0\in\Omega$. In the  proof of Proposition \ref{prop_lim}, we have shown that the sequence $u_{2,p_i}$ converges uniformly to $u_{2,\infty}$; it follows that $u_{2,p_i}-\Phi$ has a maximum at some point $x_i\in \Omega$ with $x_i\to x_0$.
	In Proposition \ref{weak_is_viscosity} it is proven that $u_{2,p_i}$ is a viscosity solution of $-\widetilde{\Delta}_{p_i} u_{2,p_i}=0$, so we obtain that
	\begin{equation*}
		-(p_i-1)\sum_{j=1}^{N}|\Phi_{x_j}|^{p_i-4}\Phi^2_{x_j}\Phi_{x_jx_j}\leq 0,
	\end{equation*}
	
	that can be rewritten as
	\begin{align*}\label{forma_comoda}
		&-(p_i-1)\left[\Norm{\nabla \Phi}{\ell^\infty}^{p_i-4}\sum_{j \in I(\nabla \Phi(x_i))}\Phi_{x_j}^2(x_i)\Phi_{x_j,x_j}(x_i)\right.\\&\left.+\sum_{j \not \in I(\nabla \Phi(x_i))}\left|\Phi_{x_j}(x_i)\right|^{p_i-4}\Phi_{x_j}(x_i)^2\Phi_{x_j,x_j}(x_i)\right]\leq 0.
	\end{align*}
	Dividing by $(p_i-1)\Norm{\nabla \Phi}{\ell^\infty}^{p_i-4}$ and passing to the limit, we obtain that $-\widetilde{\Delta}_{\infty} \Phi(x_0)\leq0$.
	Working in the same way, if $\Phi$ is touching $u$ from below in $x_0\in \Omega$, we achieve $-\widetilde{\Delta}_{\infty} \Phi(x_0)\geq0$ and then $-\widetilde{\Delta}_\infty u_{2,\infty}=0$ in the viscosity sense in $\Omega$.\\
	
	Now we deal with the boundary conditions. Let us consider $x_0\in\partial\Omega$ and $u(x_0)>0$.
	Let assume that $\Phi$  touches $u$ from below in $x_0$. Since $u_{p_i}$ converges uniformly to $u_{2,\infty}$, we have that $u_{p_i}-\Phi$ admits a minimum in some point $x_i\in\overline{\Omega}$, with $x_i\to x_0$. If $x_i\in\Omega$ for infinitely many $i$, we already have $-\widetilde{\Delta}_{\infty} \Phi(x_0)\geq 0$. 
	So, we study the case $x_i\in\partial \Omega$ ultimately for any $i$.

	If $\nabla \Phi(x_0)=0$, then $\frac{\partial\Phi}{\partial\nu}(x_0)=0$. Let  now $\nabla\Phi(x_0)\neq 0$; we have that
	\begin{equation*}
		\sum_{j=1}^{N}\left|\Phi_{x_j}(x_i)\right|^{p_i-2} 
		\Phi_{x_j}(x_i)\;\nu^j_{\partial\Omega}(x_i)\geq \Sigma_{p_i}^{p_i}(\Omega)\; |\Phi(x_i)|^{p_i-2}\Phi(x_i)\rho_{p_i}(x_i),
	\end{equation*}
	and, dividing by $\Norm{\nabla\Phi(x_i)}{\ell^\infty}^{p_i-2}$,
	
	\begin{equation}\label{right_left }
		\sum_{j=1}^{N}\left| \frac{\Phi_{x_j}(x_i)}{\Norm{\nabla\Phi(x_i)}{\ell^\infty}} \right|^{p_i-2}\Phi_{x_j}(x_i)\;\nu^j_{\partial\Omega}(x_i)\geq \Sigma_{p_i}^{p_i/(p_i-1)}(\Omega)\left| \dfrac{\Sigma_{p_i}^{p_i/(p_i-1)}(\Omega)\Phi(x_i)}{\Norm{\nabla\Phi(x_i)}{\ell^\infty}}\right|^{p_i-2}\Phi(x_i)\rho_{p_i}(x_i).
	\end{equation}
	Passing to the limit in the left-hand side, we have
	\begin{equation*}
		\lim\limits_{i \to +\infty}\sum_{j=1}^{N}\left| \frac{\Phi_{x_j}(x_i)}{\Norm{\nabla\Phi(x_i)}{\ell^\infty}} \right|^{p_i-2}\Phi_{x_j}(x_i)\;\nu^j_{\partial\Omega}(x_i)=\sum_{j \in I(\nabla \Phi(x_0))}\Phi_{x_j}(x_0)\nu^j_{\partial\Omega}(x_0).
	\end{equation*}
	From this we can deduce that the limit superior of the right-hand side in \eqref{right_left } is finite. Since $$\frac{\Sigma_{p_i}^{p_i/(p_i-1)}(\Omega)\Phi(x_i)}{\Norm{\nabla\Phi(x_i)}{\ell^\infty}} \to \frac{\Sigma_\infty(\Omega) |\Phi(x_0)| }{\Norm{\nabla\Phi(x_0)}{\ell^\infty}},$$ to have a finite limit on the right-hand side of \eqref{right_left }, we need
	$$\dfrac{\Sigma_\infty(\Omega) |\Phi(x_0)| }{\Norm{\nabla\Phi(x_0)}{\ell^\infty}} \le 1.$$
	From this last condition we have
	\begin{equation*}\label{condition}
		\Norm{\nabla\Phi(x_0)}{\ell^\infty}\geq \Sigma_\infty(\Omega) |\Phi(x_0)|\geq 0,
	\end{equation*}
	and then, taking the limit in equation \eqref{right_left },
	\begin{equation*}
		\sum_{j \in I(\nabla \Phi(x_0))}\Phi_{x_j}(x_0)\nu^j_{\partial\Omega}(x_0)\geq 0.
	\end{equation*}
	Hence, if $\Phi$ is touching $u$ from below in $x_0$,  we have
	\begin{equation}\label{final_1}
		\max\left\{ \min\left\{  \sum_{j \in I(\nabla \Phi(x_0))}\Phi_{x_j}(x_0)\nu^j_{\partial\Omega}(x_0), \Norm{\nabla\Phi(x_0)}{\ell^\infty}- \Sigma_\infty(\Omega) |\Phi(x_0)| \right\} ,-\widetilde{\Delta}\Phi(x_0) \right\}\geq 0.
	\end{equation}
	Now assume that $\Phi$ is touching $u$ from above in $x_0$. Since $u_{2,p_i}$ converges uniformly to $u_{2,\infty}$, we have that 
	$u_{2,p_i}-\Phi$  admits a maximum in some point $x_i\in\overline{\Omega}$, with $x_i\to x_0$. If $x_i\in\Omega$ for infinitely many $i$, arguing as before, we obtain $-\widetilde{\Delta }\Phi_{2,\infty}(x_0)\leq 0$. 
	If $x_i\in\partial\Omega$ ultimately for any $i$, then 
	\begin{equation*}
		\sum_{j=1}^{N}\left|\Phi_{x_j}(x_i)\right|^{p_i-2} 
		\Phi_{x_j}(x_i)\;\nu^j_{\partial\Omega}(x_i)\leq \Sigma_{p_i}^{p_i}(\Omega)\; |\Phi(x_i)|^{p_i-2}\Phi(x_i)\rho_{p_i}(x_i).
	\end{equation*}
	If $\nabla \Phi(x_0)=0$, then $\frac{\partial\Phi}{\partial\nu}(x_0)=0$; otherwise we obtain 
	\begin{equation*}
		\sum_{j=1}^{N}\left| \frac{\Phi_{x_j}(x_i)}{\Norm{\nabla\Phi(x_i)}{\ell^\infty}} \right|^{p_i-2}\Phi_{x_j}(x_i)\;\nu^j_{\partial\Omega}(x_i)\leq \Sigma_{p_i}^{p_i/(p_i-1)}(\Omega)\left| \dfrac{\Sigma_{p_i}^{p_i/(p_i-1)}(\Omega)\Phi(x_i)}{\Norm{\nabla\Phi(x_i)}{\ell^\infty}}\right|^{p_i-2}\Phi(x_i)\rho_{p_i}(x_i).
	\end{equation*}
	From this last inequality, if $\Sigma_\infty(\Omega) |\Phi(x_0)|<\Norm{\nabla\Phi(x_0)}{\ell^\infty}$, then, taking the limit,
	\begin{equation*}
		\sum_{j \in I(\nabla \Phi(x_0))}\Phi_{x_j}(x_0)\nu^j_{\partial\Omega}(x_0)\leq 0.
	\end{equation*}
	Hence,
	\begin{equation}\label{final_2}
		\min \left\{ \min\left\{  \sum_{j \in I(\nabla \Phi(x_0))}\Phi_{x_j}(x_0)\nu^j_{\partial\Omega}(x_0), \Norm{\nabla\Phi(x_0)}{\ell^\infty}- \Sigma_\infty(\Omega) |\Phi(x_0)| \right\} ,-\widetilde{\Delta}\Phi(x_0) \right\}\leq 0.
	\end{equation}
	
	Now let us suppose $u(x_0)<0$ and
	assume that $
	\Phi$ is touching $u$ from above in $x_0$.  Since $u_{2,p_i}$ converges uniformly to $u_{2,\infty}$, we have that 
	$u_{2,p_i}-\Phi$  admits a maximum at some point $x_i\in\overline{\Omega}$, with $x_i\to x_0$. If $x_i\in\Omega$ for infinitely many $i$, arguing as before, we obtain that $-\widetilde{\Delta }\Phi_{2,\infty}(x_0)\leq 0$. 
	If $x_i\in\partial\Omega$ ultimately for any $i$, then 
	\begin{equation*}
		\sum_{j=1}^{N}\left|\Phi_{x_j}(x_i)\right|^{p_i-2} 
		\Phi_{x_j}(x_i)\;\nu^j_{\partial\Omega}(x_i)\leq \Sigma^{p_i}_{p_i}(\Omega)\; |\Phi(x_i)|^{p_i-2}\Phi(x_i)\rho_{p_i}(x_i).
	\end{equation*}
	If $\nabla \Phi(x_0)=0$, then $\frac{\partial\Phi}{\partial\nu}(x_0)=0$; otherwise we obtain 
	
	\begin{equation}\label{diss}
		\sum_{j=1}^{N}\left| \frac{\Phi_{x_j}(x_i)}{\Norm{\nabla\Phi(x_i)}{\ell^\infty}} \right|^{p_i-2}\Phi_{x_j}(x_i)\;\nu^j_{\partial\Omega}(x_i)\leq \Sigma_{p_i}^{p_i/(p_i-1)}(\Omega)\left| \dfrac{\Sigma_{p_i}^{p_i/(p_i-1)}(\Omega)\Phi(x_i)}{\Norm{\nabla\Phi(x_i)}{\ell^\infty}}\right|^{p_i-2}\Phi(x_i)\rho_{p_i}(x_i).
	\end{equation}
	Now, if we pass to the limit superior on the right hand side, arguing as before and recalling this time that $\Phi(x_0)<0$, we obtain a finite quantity; this implies 
	$$\dfrac{\Sigma_\infty(\Omega) |\Phi(x_0)|}{\nabla\Phi(x_0)} \leq 1.$$
	
	Moreover, taking the limit in \eqref{diss}, since $\Phi(x_0)<0$,
	$$  \sum_{j \in I(\nabla \Phi(x_0))}\Phi_{x_j}(x_0)\nu^j_{\partial\Omega}(x_0)\leq0.$$
	Therefore,
	\begin{equation}\label{final_3}
		\min \left\{ \max\left\{  \sum_{j \in I(\nabla \Phi(x_0))}\Phi_{x_j}(x_0)\nu^j_{\partial\Omega}(x_0), -\Norm{\nabla\Phi(x_0)}{\ell^\infty}+\Sigma_\infty(\Omega) |\Phi(x_0)| \right\} ,-\widetilde{\Delta}\Phi(x_0) \right\}\leq0.
	\end{equation}
	
	Now assume that $\Phi$ is touching $u$ from below in $x_0$. Since $u_{2,p_i}$ converges uniformly to $u_{2,\infty}$, we have that 
	$u_{2,p_i}-\Phi$  admits a minimum at some point $x_i\in\overline{\Omega}$, with $x_i\to x_0$. If $x_i\in\Omega$ for infinitely many $i$, arguing as before, we obtain that $-\widetilde{\Delta }\Phi_{2,\infty}(x_0)\geq 0$. 
	If $x_i\in\partial\Omega$ ultimately for any $i$, then 
	\begin{equation*}
		\sum_{j=1}^{N}\left|\Phi_{x_j}(x_i)\right|^{p_i-2} 
		\Phi_{x_j}(x_i)\;\nu^j_{\partial\Omega}(x_i)\geq \Sigma_{p_i}^{p_i}(\Omega)\; |\Phi(x_i)|^{p_i-2}\Phi(x_i)\rho_{p_i}(x_i).
	\end{equation*}
	If $\nabla \Phi(x_0)=0$, then $\frac{\partial\Phi}{\partial\nu}(x_0)=0$; otherwise we obtain 
	
	\begin{equation}\label{dis3}
		\sum_{j=1}^{N}\left| \frac{\Phi_{x_j}(x_i)}{\Norm{\nabla\Phi(x_i)}{\ell^\infty}} \right|^{p_i-2}\Phi_{x_j}(x_i)\;\nu^j_{\partial\Omega}(x_i)\geq \Sigma_{p_i}^{p_i/(p_i-1)}(\Omega)\left| \dfrac{\Sigma_{p_i}^{p_i/(p_i-1)}(\Omega)\Phi(x_i)}{\Norm{\nabla\Phi(x_i)}{\ell^\infty}}\right|^{p_i-2}\Phi(x_i)\rho_{p_i}(x_i).
	\end{equation}
	If $\Sigma_{\infty}(\Omega)|\Phi(x_0)|<\Norm{\nabla\Phi(x_0)}{\ell^\infty}$, then, taking the limit in equation \eqref{dis3}, we achieve
	$$  \sum_{j \in I(\nabla \Phi(x_0))}\Phi_{x_j}(x_0)\nu^j_{\partial\Omega}(x_0)\geq0$$
	and consequently
	\begin{equation}\label{final_4}
		\max \left\{ \max\left\{  \sum_{j \in I(\nabla \Phi(x_0))}\Phi_{x_j}(x_0)\nu^j_{\partial\Omega}(x_0), -\Norm{\nabla\Phi(x_0)}{\ell^\infty}+ \Sigma_\infty(\Omega) |\Phi(x_0)| \right\} ,-\widetilde{\Delta}\Phi(x_0) \right\}\geq 0.
	\end{equation}

	Now let us suppose that $u(x_0)=0$ and  
	assume that $\Phi$ is touching $u$ from below in $x_0$. Since $u_{2,p_i}$ converges uniformly to $u_{2,\infty}$, we have that 
	$u_{2,p_i}-\Phi$  admits  a minimum at some point $x_i\in\overline{\Omega}$, with $x_i\to x_0$. If $x_i\in\Omega$ for infinitely many $i$, arguing as before, we obtain that $-\widetilde{\Delta }\Phi_{2,\infty}(x_0)\geq 0$. 
	If $x_i\in\partial\Omega$ ultimately for any $i$, then 
	\begin{equation*}
		\sum_{j=1}^{N}\left|\Phi_{x_j}(x_i)\right|^{p_i-2} 
		\Phi_{x_j}(x_i)\;\nu^j_{\partial\Omega}(x_i)\geq \Sigma_{p_i}^{p_i}(\Omega)\; |\Phi(x_i)|^{p_i-2}\Phi(x_i)\rho_{p_i}(x_i).
	\end{equation*}
	If $\nabla \Phi(x_0)=0$, then $\frac{\partial\Phi}{\partial\nu}(x_0)=0$; otherwise we obtain 
	
	\begin{equation*}
		\sum_{j=1}^{N}\left| \frac{\Phi_{x_j}(x_i)}{\Norm{\nabla\Phi(x_i)}{\ell^\infty}} \right|^{p_i-2}\Phi_{x_j}(x_i)\;\nu^j_{\partial\Omega}(x_i)\geq \Sigma_{p_i}^{p_i/(p_i-1)}(\Omega)\left| \dfrac{\Sigma_{p_i}^{p_i/(p_i-1)}(\Omega)\Phi(x_i)}{\Norm{\nabla\Phi(x_i)}{\ell^\infty}}\right|^{p_i-2}\Phi(x_i)\rho_{p_i}(x_i).
	\end{equation*}
	Since $0=\Sigma_{\infty}(\Omega)|\Phi(x_0)|<\Norm{\nabla\Phi(x_0)}{\ell^\infty}$, we obtain
	$$  \sum_{j \in I(\nabla \Phi(x_0))}\Phi_{x_j}(x_0)\nu^j_{\partial\Omega}(x_0)\geq0,$$
	hence
	\begin{equation}\label{final_5}
		\max \left\{  \sum_{j \in I(\nabla \Phi(x_0))}\Phi_{x_j}(x_0)\nu^j_{\partial\Omega}(x_0) ,-\widetilde{\Delta}\Phi(x_0) \right\}\geq0.
	\end{equation}
	
	Finally, assume that $\Phi$ is touching $u$ from above in $x_0$. Since $u_{2,p_i}$ converges uniformly to $u_{2,\infty}$, we have that 
	$u_{2,p_i}-\Phi$  admits a maximum at some point $x_i\in\overline{\Omega}$, with $x_i\to x_0$. If $x_i\in\Omega$ for infinitely many $i$, arguing as before, we obtain that $-\widetilde{\Delta }\Phi_{2,\infty}(x_0)\leq 0$. 
	If $x_i\in\partial\Omega$ ultimately for any $i$, then 
	\begin{equation*}
		\sum_{j=1}^{N}\left|\Phi_{x_j}(x_i)\right|^{p_i-2} 
		\Phi_{x_j}(x_i)\;\nu^j_{\partial\Omega}(x_i)\leq \Sigma_{p_i}^{p_i}(\Omega)\; |\Phi(x_i)|^{p_i-2}\Phi(x_i)\rho_{p_i}(x_i).
	\end{equation*}
	If $\nabla \Phi(x_0)=0$, then $\frac{\partial\Phi}{\partial\nu}(x_0)=0$; otherwise we obtain 
	
	\begin{equation*}
		\sum_{j=1}^{N}\left| \frac{\Phi_{x_j}(x_i)}{\Norm{\nabla\Phi(x_i)}{\ell^\infty}} \right|^{p_i-2}\Phi_{x_j}(x_i)\;\nu^j_{\partial\Omega}(x_i)\leq \Sigma_{p_i}^{p_i/(p_i-1)}(\Omega)\left| \dfrac{\Sigma_{p_i}^{p_i/(p_i-1)}(\Omega)\Phi(x_i)}{\Norm{\nabla\Phi(x_i)}{\ell^\infty}}\right|^{p_i-2}\Phi(x_i)\rho_{p_i}(x_i).
	\end{equation*}
	Since $0=\Sigma_{\infty}(\Omega)|\Phi(x_0)|<\Norm{\nabla\Phi(x_0)}{\ell^\infty}$, we obtain
	$$  \sum_{j \in I(\nabla \Phi(x_0))}\Phi_{x_j}(x_0)\nu^j_{\partial\Omega}(x_0)\leq0,$$
	hence
	\begin{equation}\label{final_6}
		\min \left\{  \sum_{j \in I(\nabla \Phi(x_0))}\Phi_{x_j}(x_0)\nu^j_{\partial\Omega}(x_0) ,-\widetilde{\Delta}\Phi(x_0) \right\}\leq0.
	\end{equation}
	
	The Theorem follows from \eqref{final_1}-\eqref{final_2}-\eqref{final_3}-\eqref{final_4}-\eqref{final_5}-\eqref{final_6}.

\end{proof}

\section{Brock-Weinstock and Weinstock type inequalities for the orthotropic $\infty$-Laplacian}
Let us denote
\begin{equation*}
	\mathcal{W}_{p}:=\{x\in\mathbb{R}^N\;|\Norm{x}{\ell^p}\leq 1\}
\end{equation*}
and, for any bounded convex set $\Omega \subseteq \mathbb{R}^N$,
\begin{equation*}
	\cP_p(\Omega):=\int_{\partial \Omega}\rho_p(x)d\cH^{N-1}(x), \quad \cM_p(\Omega):=\int_{\partial \Omega}|x|^p\rho_p(x)d\cH^{N-1}(x),
\end{equation*}
that are, respectively, the anisotropic perimeter and the boundary $p$-momentum with respect to the $\ell^p$ norm on $\R^N$.

We are interested in Brock-Weinstock and Weinstock type inequalities. In the case $p=2$, Weinstock inequality is given by
\begin{equation}\label{Wineq}
	\Sigma^2_{2}(\Omega)\cP_2(\Omega)^{\frac{1}{N-1}}\le \Sigma^2_{2}(\W_2)\cP_2(\W_2)^{\frac{1}{N-1}}
\end{equation}
and it has been proven in \cite{weinstock1954inequalities} in the case $N=2$ for simply connected sets and generalized for $N>2$ in \cite{bucur2017weinstock} (restricted to convex sets). A weaker version of this inequality, given by
\begin{equation}\label{BWineq}
	\Sigma^2_{2}(\Omega)V(\Omega)^{\frac{1}{N}}\le \Sigma^2_{2}(\W_2)V(\Omega)^{\frac{1}{N}},
\end{equation}
was proven in \cite{brock2001isoperimetric} and, for this reason, we refer to inequalities involving the Steklov eingenvalue and the volume as Brock-Weinstock inequalities.\\
A quantitative version of inequality \label{Wineq} has been achieved in \cite{gavitone2019quantitative}. Both the papers \cite{bucur2017weinstock,gavitone2019quantitative} rely on a particular isoperimetric inequality that has been generalized in \cite{paoli2019anisotropic} and that we now recall in the form that we are going to use.  Let us define the scaling invariant shape operator
\begin{equation*}
	\cI_p(\Omega):=\frac{\cM_p(\Omega)}{\cP_p(\Omega)V(\Omega)^{\frac{p}{N}}}.
\end{equation*}
Then, for any $p \in (1,\infty)$ and for any open bounded convex set $\Omega \subseteq \R^N$,  it holds
\begin{equation}\label{ineqI}
	\cI_p(\Omega)\ge \cI_p(\W_p).
\end{equation}
Moreover, let us observe that, by definition of $\W_p$ and by using the relation $NV(\W_p)=\cP_p(\W_p)$,
\begin{equation}\label{ineqeasy}
	\frac{NV(\W_p)}{\cM(\W_p)}=\frac{\cP_p(\W_p)}{\cP_p(\W_p)}=1.
\end{equation}
In the general case of the orthotropic $p$-Laplacian, the following Brock-Weinstock type inequality (restricted to bounded convex open sets) has been proven in \cite{brasco2013anisotropic}
\begin{equation}\label{BFBWineq}
	\Sigma_p^p(\Omega)V(\Omega)^{\frac{p-1}{N}}\le V(\W_p)^{\frac{p-1}{N}}.
\end{equation}
As a first step, we want to improve the previous inequality, to include in some way the perimeter.
\begin{thm}\label{thmmain}
	Let $\Omega \subset \R^N$ be an open bounded convex set and $p>1$. Consider $q\ge 0$ and $r \in [0,N]$ such that $\frac{p}{N}=q+\frac{r}{N}$.
	Then, we have
	\begin{equation}\label{eq:step}
		\Sigma_{p}^p(\Omega)\cP_p(\Omega)^{\frac{r-1}{N-1}}V(\Omega)^{q}\le\cP_p(\W_p)^{\frac{r-1}{N-1}}V(\W_p)^{q}.
	\end{equation} 
\end{thm}
\begin{proof}
	Let us first recall that by \cite[Lemma $7.1$]{brasco2013anisotropic}, we can use the functions $x_i$ with $i=1,\dots,N$ as test functions in the Rayleigh quotient $\cR_p$ for $\Sigma_{p}^p$ (up to a rigid movement of $\Omega$), with
	\begin{equation*}
		\cR_p[x_i]=\frac{V(\Omega)}{\int_{\partial \Omega}|x_i|^p\rho_p(x)d\cH^{N-1}(x)},
	\end{equation*}
	hence, for any $i=1,\dots,N$,  we have
	\begin{equation*}
		\Sigma_{p}^p(\Omega)\int_{\partial \Omega}|x_i|^p\rho_p(x)d\cH^{N-1}(x)\le V(\Omega).
	\end{equation*}
	Summing over $i$ we have
	\begin{equation}\label{pass5.2.1}
		\Sigma_{p}^p(\Omega)\le \frac{NV(\Omega)}{\cM_p(\Omega)}.
	\end{equation}
	Now let us write inequality \eqref{ineqI} explicitly to achieve
	\begin{equation*}
		\frac{\cM_p(\Omega)}{\cP_p(\Omega)V(\Omega)^{\frac{p}{N}}}\ge \frac{\cM_p(\W_p)}{\cP_p(\W_p)V(\W_p)^{\frac{p}{N}}}
	\end{equation*}
	and then
	\begin{equation*} 
		\cM_p(\Omega)\ge \frac{\cM_p(\W_p)\cP_p(\Omega)V(\Omega)^{\frac{p}{N}}}{\cP_p(\W_p)V(\W_p)^{\frac{p}{N}}}.
	\end{equation*}
	Using this inequality in equation \eqref{pass5.2.1} we get
	\begin{equation*}\label{pass3.2.5}
		\Sigma_{p}^p(\Omega)\le \frac{NV(\Omega)\cP_p(\W_p)V(\W_p)^{\frac{p}{N}}}{\cM_p(\W_p)\cP_p(\Omega)V(\Omega)^{\frac{p}{N}}},
	\end{equation*}
	that can be recast as
	\begin{align*}
		\Sigma_{p}^p(\Omega)\le \frac{N\cP_p(\W_p)V(\W_p)^{\frac{p}{N}}}{\cM_p(\W_p)\cP_p(\Omega)^{\frac{r-1}{N-1}}V^{q}(\Omega)}\left(\frac{V(\Omega)^{1-\frac{1}{N}}}{\cP_p(\Omega)}\right)^{\frac{N-r}{N-1}}.
	\end{align*}
	Let us recall the anisotropic standard isoperimetric inequality (see \cite[Proposition $2.3$]{alvino1997convex}):
	\begin{equation*}
		\frac{V(\Omega)^{1-\frac{1}{N}}}{\cP_p(\Omega)}\le \frac{V(\W_p)^{1-\frac{1}{N}}}{\cP_p(\W_p)}.
	\end{equation*}
	Thus, since $r < N$ and then $\frac{N-r}{N-1}>0$, we have
	\begin{equation*}
		\Sigma_{p}^p(\Omega)\le \frac{N\cP_p(\W_p)V(\W_p)^{\frac{p}{N}}}{\cM_p(\W_p)\cP_p(\Omega)^{\frac{r-1}{N-1}}V(\Omega)^q}\left(\frac{V(\W_p)^{1-\frac{1}{N}}}{\cP_p(\W_p)}\right)^{\frac{N-r}{N-1}}
	\end{equation*}
	and then, recalling that $p/N=q+r/N$, we finally get
	\begin{equation*}\label{pass2.5.2}
		\Sigma^p_{p}(\Omega)\cP_p(\Omega)^{\frac{r-1}{N-1}}V(\Omega)^q\le \frac{N V(\W_p)}{\cM_p(\W_p)}\cP_p(\W_p)^{\frac{r-1}{N-1}}V(\W_p)^q.
	\end{equation*}
	Equality \eqref{ineqeasy} concludes the proof.\\
\end{proof}
\begin{rmk}
	Let us observe that Theorem \ref{thmmain} includes inequality \eqref{BFBWineq}. Indeed, for any bounded convex set $\Omega$ and any $p>1$ we can fix $r=1$ and then $q=\frac{p-1}{N}$ in inequality \eqref{eq:step} to obtain the desired result.\\
	Moreover, let us observe that, in general, inequality \eqref{eq:step} implies inequality \eqref{BFBWineq}. Indeed, since the left-hand side of equation \eqref{eq:step} is scaling invariant, we can always suppose $\cP_p(\Omega)=\cP_p(\W_p)$. Thus, the aforementioned equation becomes
	\begin{equation*}
		\Sigma_{p}^p(\Omega)V^q(\Omega)\le V^q(\W_p).
	\end{equation*}
	Multiplying both sides by $V^\frac{p-r}{N}(\Omega)$ we have
	\begin{equation*}
		\Sigma_{p}^p(\Omega)V^\frac{p-1}{N}(\Omega)\le V^q(\W_p)V^\frac{p-r}{N}(\Omega)\le V^\frac{p-1}{N}(\W_p),
	\end{equation*}
	where the last inequality follows from the anisotropic isoperimetric inequality.\\
	As in \cite{brasco2013anisotropic}, we are not able to detect equality cases. However, let us stress out that equality could not hold even for $\W_p$ if $\Sigma_p^p(\W_p)<1$. Let us recall that in general it is known that $\Sigma_p^p(\W_p)\le 1$, but determining if it is actually equal to $1$ or not is still an open problem, except that for $p=2$.
\end{rmk}
An improvement that involves only the perimeter can be shown if $p \le N$. Indeed, we have the following Corollary.
\begin{cor}
	Let $\Omega \subset \R^N$ be an open bounded convex set and $p \in (1,N]$.
	Then, we have
	\begin{equation*}
		\Sigma^p_{p}(\Omega)\cP_p(\Omega)^{\frac{p-1}{N-1}}\le\cP_p(\W_p)^{\frac{p-1}{N-1}}.
	\end{equation*}
\end{cor}
\begin{proof}
	Just observe that if $p \in (1,N]$, we can choose $r=p$ and $q=0$ in equation \eqref{eq:step}.
\end{proof}
\begin{rmk}
	If the conjecture by Brasco and Franzina in \cite{brasco2013anisotropic} reveals to be true, i. e. the fact that $\Sigma^p_{p}(\W_p)=1$, last result implies the Weinstock inequality for the orthotropic $p$-Laplacian as $p \in (1,N]$.
\end{rmk}
In any case, we can recast equation \eqref{BFBWineq} as
\begin{equation*}
	\Sigma_p(\Omega)V^{\frac{p-1}{Np}}(\Omega) \le V^{\frac{p-1}{Np}}(\W_p)
\end{equation*}
and then take the limit as $p \to +\infty$ to obtain
\begin{equation}\label{ineq:BWw}
	\Sigma_{\infty}(\Omega)V(\Omega)^{\frac{1}{N}}\le V(\W_\infty)^{\frac{1}{N}},
\end{equation}
that cannot be rewritten in a full scaling-invariant form since $\Sigma_{\infty}(\W_\infty)=1/N$. Moreover, being $V(\W_\infty)=2^N$, equation \eqref{ineq:BWw} can be rewritten as
\begin{equation*}
	\Sigma_{\infty}(\Omega)V(\Omega)^{\frac{1}{N}}\le 2.
\end{equation*}
However, we can improve such inequality by means of an anisotropic isodiametric inequality.
\begin{cor}
	For any bounded convex open set $\Omega \subset \R^N$ it holds
	\begin{equation}\label{eq:isodiam}
		\Sigma_{\infty}(\Omega)V(\Omega)^{1/N}\leq \Sigma_{\infty}(\W_1)V(\W_1)^{1/N}.
	\end{equation}
	Equality holds if and only if $\Omega$ is equivalent to $\W_1$ up to translations and scalings.
\end{cor}
\begin{proof} 
	Let us observe that, by \cite[Proposition $2.1$]{piscitelli2019anisotropic}, we have
	\begin{equation*}
		\frac{2}{{\rm diam}_1(\Omega)}V^{1/N}(\Omega)\leq V^{1/N}(\W_1) 
	\end{equation*}
	Recalling that $\Sigma_\infty(\Omega)=\frac{2}{{\rm diam}_1(\Omega)}$ and $\Sigma_\infty(\W_1)=\frac{2}{{\rm diam}_1(\W_1)}=1$ we conclude the proof. Equality cases follow from \cite[Proposition $2.1$]{piscitelli2019anisotropic}.
\end{proof}
\begin{rmk}
	Let us observe that inequality \eqref{eq:isodiam} implies inequality \eqref{ineq:BWw}, since $V(\W_1)=2^{\frac{N}{2}}$.
\end{rmk}
On the other hand, a Weinstock-type inequality in the planar case follows from the Rosenthal-Szasz inequality in Radon planes (see \cite{balestro2019rosenthal}). To give this result we need to introduce the concept of width in our case. Fix $N=2$ and consider any bounded open convex set $\Omega$. For each direction $v$ there exists two supporting lines $r_1,r_2$ for $\Omega$ that are orthogonal to $v$ in the Euclidean sense. We call width of $\Omega$ in the direction $v$ the distance $\omega(v)=d_1(r_1,r_2)$. With this in mind, we can give the following result.
\begin{cor}
	For any  open bounded convex set $\Omega\subseteq \R^2$ it holds
	\begin{equation}\label{ineq:Ws}
		\Sigma_{\infty}(\Omega)P_\infty(\Omega)\le \Sigma_\infty(\W_1)P_\infty(\W_1).
	\end{equation}
	Equality holds if and only if $\Omega$ is of constant width, i.e. if and only if $\omega(v)\equiv \diam_1(\Omega)$. 
\end{cor}
\begin{proof}
	Let us recall that the Rosenthal-Szasz inequality for Radon planes \cite[Theorem $1.1$]{balestro2019rosenthal}, specified to the plane $(\R^2,\Norm{\cdot}{\ell^1})$, is given by
	\begin{equation*}
		\frac{2 P_\infty(\Omega)}{\diam_1(\Omega)}\le P_\infty(\W_1). 
	\end{equation*}
	Recalling that $\Sigma_\infty(\W_1)=1$ and $\Sigma_\infty(\Omega)=\frac{2}{\diam_1(\Omega)}$ we conclude the proof. Equality cases follow from equality cases of the Rosenthal-Szasz inequality in \cite[Theorem $1.1$]{balestro2019rosenthal}.
\end{proof}
\section*{Acknowledgments}
We would like to thank the referee for his/her really useful suggestions to improve the paper.
\bibliographystyle{abbrv}
\bibliography{biblio}

\begin{thebibliography}{10}

\bibitem{alvino1997convex}
A.~Alvino, V.~Ferone, G.~Trombetti, and P.-L. Lions.
\newblock Convex symmetrization and applications.
\newblock {\em Annales de l'Institut Henri Poincare (C) Non Linear Analysis},
  14(2):275--293, 1997.

\bibitem{aronsson1967extension}
G.~Aronsson.
\newblock Extension of functions satisfying {L}ipschitz conditions.
\newblock {\em Arkiv f\"or Matematik}, 6(6):551--561, 1967.

\bibitem{balestro2019rosenthal}
V.~Balestro and H.~Martini.
\newblock The {R}osenthal--{S}zasz inequality for normed planes.
\newblock {\em Bulletin of the Australian Mathematical Society},
  99(1):130--136, 2019.

\bibitem{belloni2004pseudo}
M.~Belloni and B.~Kawohl.
\newblock The pseudo-$p$-{L}aplace eigenvalue problem and viscosity solutions
  as $p\to\infty$.
\newblock {\em ESAIM: Control, Optimisation and Calculus of Variations},
  10(1):28--52, 2004.

\bibitem{bognar1993lower}
G.~Bogn{\'a}r.
\newblock A lower bound for the smallest eigenvalue of some nonlinear
  eigenvalue problems on convex domains in two dimensions.
\newblock {\em Applicable Analysis}, 51(1-4):277--288, 1993.

\bibitem{bognar2004isoperimetric}
G.~Bogn{\'a}r.
\newblock Isoperimetric inequalities for some nonlinear eigenvalue problems.
\newblock {\em Electronic Journal of Qualitative Theory of Differential
  Equations}, (4):1--12, 2004.

\bibitem{bousquet2018c1}
P.~Bousquet and L.~Brasco.
\newblock {$C^1$ }regularity of orthotropic p-harmonic functions in the plane.
\newblock {\em Anal. PDE}, 11(4):813--854, 2018.

\bibitem{bousquet2018lipschitz}
P.~Bousquet, L.~Brasco, C.~Leone, and A.~Verde.
\newblock On the {L}ipschitz character of orthotropic p-harmonic functions.
\newblock {\em Calculus of Variations and Partial Differential Equations},
  57(3):88, 2018.

\bibitem{brasco2012spectral}
L.~Brasco, G.~De~Philippis, and B.~Ruffini.
\newblock Spectral optimization for the {S}tekloff--{L}aplacian: the stability
  issue.
\newblock {\em Journal of Functional Analysis}, 262(11):4675--4710, 2012.

\bibitem{brasco2013anisotropic}
L.~Brasco and G.~Franzina.
\newblock An anisotropic eigenvalue problem of {S}tekloff type and weighted
  {W}ulff inequalities.
\newblock {\em Nonlinear Differential Equations and Applications NoDEA},
  20(6):1795--1830, 2013.

\bibitem{brock2001isoperimetric}
F.~Brock.
\newblock An isoperimetric inequality for eigenvalues of the {S}tekloff
  problem.
\newblock {\em ZAMM-Journal of Applied Mathematics and Mechanics/Zeitschrift
  f{\"u}r Angewandte Mathematik und Mechanik: Applied Mathematics and
  Mechanics}, 81(1):69--71, 2001.

\bibitem{bucur2017weinstock}
D.~Bucur, V.~Ferone, C.~Nitsch, and C.~Trombetti.
\newblock Weinstock inequality in higher dimensions.
\newblock {\em arXiv preprint arXiv:1710.04587}, 2017.

\bibitem{chambolle2019existence}
A.~Chambolle, M.~Morini, M.~Novaga, and M.~Ponsiglione.
\newblock Existence and uniqueness for anisotropic and crystalline mean
  curvature flows.
\newblock {\em Journal of the American Mathematical Society}, 32(3):779--824,
  2019.

\bibitem{espositoNeumann}
L.~Esposito, B.~Kawohl, C.~Nitsch, and C.~Trombetti.
\newblock The {N}eumann eigenvalue problem for the $\infty$-{L}aplacian.
\newblock {\em Atti Accad. Naz. Lincei Rend. Lincei Mat. Appl.}, 26:119--134,
  2015.

\bibitem{garcia2006steklov}
J.~Garc{\'\i}a-Azorero, J.~J. Manfredi, I.~Peral, and J.~D. Rossi.
\newblock Steklov eigenvalues for the $\infty$-{L}aplacian.
\newblock {\em Rendiconti Lincei-Matematica e Applicazioni}, 17(3):199--210,
  2006.

\bibitem{gavitone2019quantitative}
N.~Gavitone, D.~A. La~Manna, G.~Paoli, and L.~Trani.
\newblock A quantitative {W}einstock inequality for convex sets.
\newblock {\em Calculus of Variations and Partial Differential Equations},
  59(1):2, 2020.

\bibitem{katzourakis2014introduction}
N.~Katzourakis.
\newblock {\em An Introduction to Viscosity Solutions for Fully Nonlinear PDE
  with Applications to Calculus of Variations in $L^\infty$}.
\newblock Springer, 2014.

\bibitem{lindqvist2016notes}
P.~Lindqvist.
\newblock {\em Notes on the Infinity Laplace Equation}.
\newblock Springer, 2016.

\bibitem{lindqvist2017notes}
P.~Lindqvist.
\newblock {\em Notes on the $p$-Laplace Equation}.
\newblock University of Jyv{\"a}skyl{\"a}, 2017.

\bibitem{lindqvist2018regularity}
P.~Lindqvist and D.~Ricciotti.
\newblock Regularity for an anisotropic equation in the plane.
\newblock {\em Nonlinear Analysis}, 177:628--636, 2018.

\bibitem{Lions}
J.~L. Lions.
\newblock {\em Quelques m\'{e}thodes de r\'{e}solutions des problèmes aux
  limites non lin\'{e}aires}.
\newblock Dunod, Gauthier-Villars, Paris, 1969.

\bibitem{paoli2019anisotropic}
G.~Paoli and L.~Trani.
\newblock Anisotropic isoperimetric inequalities involving boundary momentum,
  perimeter and volume.
\newblock {\em Nonlinear Analysis}, 187:229--246, 2019.

\bibitem{piscitelli2019anisotropic}
G.~Piscitelli.
\newblock The anisotropic $\infty$-{L}aplacian eigenvalue problem with
  {N}eumann boundary conditions.
\newblock {\em Differential and Integral Equations}, 32(11/12):705--734, 2019.

\bibitem{rossi2007optimal}
J.~D. Rossi and M.~Saez.
\newblock Optimal regularity for the pseudo infinity {L}aplacian.
\newblock {\em ESAIM: Control, Optimisation and Calculus of Variations},
  13(2):294--304, 2007.

\bibitem{rossi_saintier}
J.~D. Rossi and N.~Saintier.
\newblock The limit as $p \to +\infty$ of the first eigenvalue for the
  $p$-{L}aplacian with mixed {D}irichlet and {R}obin boundary conditions.
\newblock {\em Nonlinear Analysis: Theory, Methods \& Applications},
  119:167--178, 2015.

\bibitem{savare1996regularity}
G.~Savar{\'e}.
\newblock On the regularity of the positive part of functions.
\newblock {\em Nonlinear Analysis: Theory, Methods \& Applications},
  27(9):1055--1074, 1996.

\bibitem{Visik}
I.~Vi\v{s}ik.
\newblock Sur la r\'{e}solutions des probl\`{e}mes aux limites pour des
  \'{e}quations paraboliques quasi-lin\`{e}aires d'ordre quelconque.
\newblock {\em Mat. Sbornik}, 59:289--235, 1962.

\bibitem{Visik2}
I.~Vi\v{s}ik.
\newblock Quasilinear strongly elliptic systems of differential equations in
  divergence form.
\newblock {\em Trans. Moscow. Math. Soc. 12}, 12:140--208, 1963.

\bibitem{weinstock1954inequalities}
R.~Weinstock.
\newblock Inequalities for a classical eigenvalue problem.
\newblock {\em Journal of Rational Mechanics and Analysis}, 3:745--753, 1954.

\end{thebibliography}
\end{document}